\documentclass[onecolumn]{IEEEtran}
\usepackage{amsmath,amsfonts,amssymb,euscript, graphicx, 
epsfig,
enumerate,float,afterpage, subfigure, ifthen}%

\usepackage{url}
\usepackage{hyperref}

\newtheorem{thm}{Theorem}

\newtheorem{lem}{Lemma}

\newcommand{\expect}[1]{\mathbb{E}\left[#1\right]}

\newcommand{\script}[1]{{{\cal{#1} }}}

\newcommand{\transpose}{\top}
\newcommand{\given}{\left\vert\right.}
\newcommand{\Conv}{\mbox{Conv}}

\begin{document}

\title
  {A Converse Result on Convergence Time  for Opportunistic Wireless Scheduling}
\author{Michael J. Neely \\ University of Southern California\\ \url{https://viterbi-web.usc.edu/~mjneely/}
\thanks{A partial version of this work was accepted for presentation at the IEEE INFOCOM conference, 2020 \cite{neely-converse-opportunistic-infocom2020}.} 
\thanks{This work was supported in part by one or move of: NSF
CCF-1718477, NSF SpecEES 1824418.}
}

\markboth{}{Neely}

\maketitle

\begin{abstract} 
This paper proves an impossibility result for 
stochastic network utility maximization for multi-user wireless systems, including multiple access and broadcast systems.   Every time slot an
access point observes the current channel states for each user and 
opportunistically selects a vector of transmission rates.  Channel state vectors are assumed to be independent and identically distributed with an unknown probability distribution.   The goal is to learn to make decisions over time that maximize a concave utility function of the running 
time average transmission rate of each user. Recently it was shown that a stochastic Frank-Wolfe algorithm converges to utility-optimality with an error of $O(\log(T)/T)$, where $T$ is the time the algorithm has been running.  An existing $\Omega(1/T)$ converse 
is known. The current paper improves the converse to $\Omega(\log(T)/T)$, which matches the known achievability result.  It does this by constructing a particular (simple) system for which no algorithm can achieve a better performance. The proof uses a novel reduction of the opportunistic scheduling problem to a problem of estimating a Bernoulli probability $p$ from independent and identically distributed samples. Along the way we refine a regret bound for Bernoulli estimation to show that, for any sequence of estimators, the set of values $p \in [0,1]$ under which the estimators perform poorly has measure at least $1/6$.  
\end{abstract}

\section{Introduction} 

This paper establishes the fundamental learning rate for network utility maximization in 
wireless opportunistic scheduling systems, such as multiple access systems 
and broadcast systems.  The recent work \cite{neely-frank-wolfe-ton} 
shows that a stochastic Frank-Wofe algorithm 
with a vanishing stepsize achieves a utility optimality gap that decays like $O(\log(T)/T)$, where 
$T$ is the time the algorithm is in operation.  It does this without a-priori knowledge of the channel state probabilities. This paper establishes a matching converse. A simple example 
system is constructed for which 
all algorithms have an error gap of at least $\Omega(\log(T)/T)$.    Specifically, 
we construct a system with channel states parameterized by an unknown probability $q \in [0,1]$ such that for any algorithm, there is a set $\script{Q} \subseteq [0,1]$ with measure at least $1/6$ under which the algorithm performs poorly.  This is done by a novel reduction of the opportunistic scheduling problem to a problem of estimating a Bernoulli probability $p$ from independent and identically distributed (i.i.d.)  Bernoulli samples.  Along the way, a refined statement regarding the regret of Bernoulli estimation is developed. 

A general structure for the class of opportunistic scheduling systems is as follows:  
The system is assumed to operate over slotted time $t \in \{0, 1, 2, \ldots\}$.  There are $n$ users. 
Every slot $t \in \{0, 1, 2, \ldots\}$ an access point allocates a vector $X[t]=(X_1[t], \ldots, X_n[t])$ for transmission of independent data belonging to each user.  In the case of 
wireless multiple access systems, the $n$ users transmit their data over uplink channels to the access point. It is assumed they  use a coordinated scheme that allows successful decoding of all transmissions at the scheduled bit rates $X[t]$. In the case of wireless broadcast systems, the access point transmits data for each user over downlink channels 
at the scheduled bit rates $X[t]$. 

The set of all transmission rate
vectors that are available on a particular slot $t$ can change 
from one slot to the next.  This can arise from time-varying connection properties
such as channel states that vary due to device 
mobility.
We model this time-variation by a \emph{random state vector} $S[t] \in \mathbb{R}^m$ that is observed by the access point at the start of every slot $t$ (where $m$ is a positive integer that can be different from $n$).   Assume that  $\{S[t]\}_{t=0}^{\infty}$ is i.i.d. over slots with some distribution $F_{S}(s)=P[S[0]\leq s]$ for all $s \in \mathbb{R}^m$ (where inequality is taken entrywise).  The distribution function $F_S(s)$ is unknown.   Define $\script{D}(S[t])$ as the \emph{decision set} for slot $t$, being the set of all $(X_1[t], \ldots, X_n[t])$ vectors that can be chosen on slot $t$ when the channel state vector is $S[t]$.   

The structure of $\script{D}(S[t])$ depends on the network.   For example, a multiple access network might allow only one user to transmit per slot. In this case we 
can define $S[t]=(S_1[t], \ldots, S_n[t])$ as a vector of \emph{channel states}, 
where $S_i[t]$ represents the transmission rate available to user $i$ on slot $t$ if that user is selected for transmission.  
Then $\script{D}(S[t])$ is a set that contains $n$ vectors:  
$$ \script{D}(S[t]) = \{(S_1[t],0,0,...,0), (0,S_2[t], 0, ..., 0), ..., (0,0,...,0,S_n[t))\}$$
where the $i$th vector in this set corresponds to choosing user $i$ for transmission. 
More sophisticated wireless signaling schemes  
can allow the set $\script{D}(S[t])$ to contain vectors with multiple positive components. The 
set $\script{D}(S[t])$ can be uncountably infinite in the case when transmission rates depend on an uncountably infinite set of power allocation levels that are available for scheduling. 

Every slot $t \in \{0, 1, 2, \ldots\}$ the access point observes $S[t]$ 
and chooses $X[t] \in \script{D}(S[t])$ in such a way that, over time, 
the following problem is solved: 
\begin{align}
&\mbox{Maximize:} \quad \liminf_{T\rightarrow\infty} \phi\left(\frac{1}{T}\sum_{t=0}^{T-1}\expect{X[t]}\right) \label{eq:p1} \\
&\mbox{Subject to:} \quad X[t] \in \script{D}(S[t]) \quad \forall t \in \{0, 1, 2, \ldots\}\label{eq:p2} 
\end{align}
where $\phi(x_1, \ldots, x_n)$ is a given real-valued utility function of 
the average user transmission rates. The function $\phi$ is assumed to be concave and entrywise nondecreasing.   Let $\phi^*$ be the optimal utility value, which  considers all possible algorithms that operate over an infinite time horizon, including algorithms that have perfect knowledge of the probability distribution $F_S$, and even including \emph{non-causal} algorithms that have knowledge of future states $\{S[t]\}_{t=0}^{\infty}$.\footnote{The only assumption about the algorithms over which $\phi^*$ is optimized is that the algorithms are \emph{probabilistically measurable} so that they produce random vectors $X[t]$ with well defined expectations on all slots $t$, see \cite{sno-text} for details. This is a mild assumption. It precludes (impractical) algorithms that make decisions using the axiom of choice on every slot $t$ to produce non-measurable $X[t]$ vectors.}     It is challenging to design a (causal) 
scheduling algorithm that achieves 
a utility that is close to $\phi^*$, particularly when the distribution $F_S$ is unknown to the network controller.   Algorithms
that are causal (so that they have no knowledge of the future) 
and that have no a-priori knowledge of the distribution $F_S$ shall be called \emph{statistics-unaware} algorithms.  This paper establishes the fundamental convergence delay required for any statistics-unaware algorithm to achieve utility that is close to the optimal value $\phi^*$.

A general statistics-unaware algorithm may incorporate some type of learning or estimation of the distribution $F_S$ or some functional of this distribution.  Observations of past channel states can be exploited when making online decisions. Consider some statistics-unaware algorithm that  
makes decisions over time $t \in \{0, 1, 2,\ldots\}$.  For each positive integer $T$, the expression 
$$ \phi\left(\frac{1}{T}\sum_{t=0}^{T-1} \expect{X[t]}\right)$$ 
is the utility associated with running the algorithm over the first $T$ slots $\{0, 1, 2, \ldots, T-1\}$. 
This utility includes decisions $X[t]$ made at each step of the way (including the decision $X[0]$ made at time $t=0$ based only on the observation $S[0]$). Decisions must be made intelligently at each step of the way and fast learning is crucial. How close can the achieved utility  get to the optimal value $\phi^*$?  What time $T$ is required?

\subsection{Example utility functions} \label{section:linear}

Different concave utility functions can be used to provide different types of performance (with corresponding fairness properties).  For example, consider the linear utility 
$$ \phi(x_1, \ldots, x_n) = \sum_{i=1}^n a_ix_i$$
where $a_1, \ldots, a_n$ are given nonnegative weights.  Under this utility function, the problem \eqref{eq:p1}-\eqref{eq:p2} seeks to maximize a weighted sum of average transmission rates of each user.   This linear utility is a trivial special case: The statistics-unaware algorithm of observing $S[t]$ at the start of every slot $t$ and 
choosing $X[t] \in \script{D}(S[t])$ to greedily maximize $\phi(X[t])$, called the \emph{greedy algorithm},  can be shown to lead to immediate convergence.\footnote{The existence of a maximizer for $\phi(x)$ over all $x \in \script{D}(S[t])$ holds under the mild additional assumption that $\phi$ is continuous and $\script{D}(S[t])$ is  a 
compact subset of $\mathbb{R}^n$ for each slot $t$.} This is because the time average expectation can be pushed inside the linear function $\phi$ and so maximization of immediate rewards translates into maximization of long term rewards. 

This is not the case for concave but \emph{nonlinear} utility functions.  This is because 
the goal is to maximize a concave function of the time average, not to maximize the average of a concave function.  This goal is crucial to network fairness.
The greedy algorithm can be far from optimal for general concave but nonlinear utility functions.\footnote{Consider a 2-user example with utility $\phi(x_1,x_2)=\log(x_1)+\log(x_2)$, a function known to have \emph{proportional fairness} properties \cite{kelly-charging}\cite{kelly-shadowprice}.  Suppose there is no channel variation and  $\script{D}(S[t]) = \{(20,0), (0,19)\}$ for all $t$. 
Choosing a fraction of time $p$ to transmit with user $1$ and maximizing
$\log(20p) + \log(19(1-p))$, which is the utility function applied to the long term transmission rate vector,  leads to $p^*=1/2$ and yields a long term transmission
rate vector $\left(\overline{X}_1, \overline{X}_2\right) = (10, 9.5)$. In contrast, the greedy strategy chooses user 1 on every slot so that $\left(\overline{X}_1, \overline{X}_2\right) = (20,0)$ and  user 2 never gets a chance to transmit!} 

From a fairness perspective, linear utilities are undesirable.  For example, suppose there are two users, at most one user can transmit per slot, and user 1 always has a strictly better channel condition than user 2 (perhaps because user 1 is closer to the access point).  Maximizing the linear utility function $\phi(x_1,x_2)=x_1+x_2$  results in the algorithm that  \emph{always} chooses user 1, so that user 2 receives a time average rate of 0. One way to be fair to user 2 is to change the utility function to 
$$ \phi(x_1,x_2) = \min[x_1,x_2]$$
Under this concave (but non-smooth) utility function, the problem is to maximize the minimum average rate given to the users. Another common type of (smooth) utility function is 
$$ \phi(x_1, x_2) = \log(x_1) + \log(x_2) $$
This logarithmic utility function results in a type of fairness called 
\emph{proportional fairness} \cite{kelly-charging}\cite{kelly-shadowprice}.  The logarithmic utility function is often modified to remove the singularity at zero: 
$$ \phi(x_1,x_2) = \log(1 + c x_1) + \log(1+c x_2) $$
where $c>0$ is a constant. Large values of $c$ can be used to 
approximate the $\log(x_1)+\log(x_2)$ function. Other types of concave and nonlinear
utility functions can be used for other types of fairness, such as $\alpha$-fair utility 
functions \cite{mo-walrand-fair}\cite{altman-alpha-fair}\cite{chiang-axiomatic-fairness}. 

\subsection{Prior work}

The work  \cite{prop-fair-down}\cite{vijay-allerton02} develops statistics-unaware
Frank-Wolfe type algorithms (with various step size rules) 
for solving the problem \eqref{eq:p1}-\eqref{eq:p2} for smooth utility functions using a fluid limit analysis.  An alternative statistics-unaware 
drift-plus-penalty algorithm of \cite{neely-fairness-ton}\cite{sno-text} 
can be used to solve \eqref{eq:p1}-\eqref{eq:p2} for smooth or 
nonsmooth utility functions, and this achieves utility within $\epsilon$ of optimality 
with convergence time $O(1/\epsilon^2)$. 
Drift-plus-penalty can also be used for extended problems of multi-hop networks with power minimization and constraints \cite{neely-energy-it}, and related algorithms for these extended problems are in \cite{atilla-primal-dual-jsac}\cite{atilla-fairness-ton}\cite{stolyar-greedy}\cite{stolyar-gpd-gen}. 

 Recent work in \cite{neely-frank-wolfe-ton} shows that, for smooth utility functions, 
a Frank-Wolfe algorithm with a constant stepsize also has convergence time $O(1/\epsilon^2)$, while a Frank-Wolfe algorithm with a 
vanishing stepsize yields an improved $O(\log(1/\epsilon)/\epsilon)$ convergence time. In particular, in the latter case we obtain 
$$  \phi\left(\frac{1}{T}\sum_{t=0}^{T-1} \expect{X[t]}\right) \geq \phi^* - \frac{c\log(T)}{T} \quad \forall T \in \{1, 2, 3, ...\}$$ 
where $c>0$ is a particular system constant.  The work \cite{neely-frank-wolfe-ton} 
also provides a near-matching converse of $\Omega(1/T)$. The problem of closing the logarithmic gap between the achievability bound and the converse bound was left as an open question. We resolve that open question in this paper by showing that the $O(\log(T)/T)$ gap is optimal. 

A related logarithmic convergence time result is developed by Hazan and Kale in 
 \cite{hazan-kale-stochastic} for the context of online convex optimization with strongly convex objective functions.  The prior work \cite{hazan-kale-stochastic} provides an example online convex optimization problem that immediately reduces to a problem of estimating a Bernoulli probability from i.i.d. Bernoulli samples.  They then provide a deep analysis of the Bernoulli estimation problem to show, via a nested interval argument, that for any sequence of Bernoulli estimators there exists a probability $p \in [1/4, 3/4]$ under which the estimators have a sum mean square error that grows at least logarithmically in the number of samples.  This prior work inspires the current paper.  We show that certain opportunistic scheduling problems can also be reduced to Bernoulli estimation; then we can use the Bernoulli estimation result of \cite{hazan-kale-stochastic}. However, this reduction is not obvious.  Online convex optimization problems have a different structure than opportunistic scheduling problems and the same reduction techniques cannot be used.  New techniques are used to establish the converse, including a novel reduction of the opportunistic scheduling problem to a Bernoulli estimation problem.  
 
 A conceptually related problem on the fundamental time required to minimize  a deterministic convex function with noisy observations of the gradients/subgradients is treated in \cite{hazan-kale-stochastic} (where an achievability result is developed concerning the error versus time)  and from an information theoretic perspective in \cite{best-stochastic} (where converse results are developed); see also early work on computational complexity for this problem in \cite{Nem-Yudin}. 
 At a high level, the convergence time and learning concepts for that problem are similar to the current paper.  However, the structure and analysis of that problem is quite different. For example:  (i) Unlike the problem of this paper, the fundamental asymptotic tradeoffs for that problem depend significantly on whether or not the convex function is strongly convex; (ii) The fundamental tradeoffs for that problem do not have the same logarithmic properties as the problem of the current paper.   
 
 Prior work on fundamental bounds on the time and accuracy for 
 estimation problems  is in, for example, \cite{risk-bounds-learning} for classification 
 problems, and work in  \cite{jiao-minimax}\cite{jiao-minimax2} treats bounds for 
 estimation of functionals of discrete distributions.

 \subsection{Our contributions}

\begin{enumerate} 

\item This paper proves an $\Omega(\log(T)/T)$ converse  for opportunistic scheduling. This matches an existing achievability result and resolves the open question in \cite{neely-frank-wolfe-ton} to show that this performance is optimal. 

  \item This paper shows that \emph{strongly concave} utility functions 
 cannot be used to improve the asymptotic 
 convergence time for opportunistic scheduling problems in comparison to  functions that are concave but not strongly concave. This is surprising because strong convexity/concavity  provides convergence improvements in other contexts, including 
 online convex optimization problems \cite{kale-universal}\cite{online-convex}, deterministic  minimization via 
 gradient descent  \cite{nesterov-book}, and deterministic minimization via stochastic gradients \cite{best-stochastic}\cite{Nem-Yudin}\cite{hazan-kale-stochastic}.   This emphasizes the unique properties of 
 opportunistic scheduling problems.
 
 \item The technique for reducing opportunistic scheduling to Bernoulli estimation 
 can more broadly impact future work on more complex networks (see open questions in this direction in the conclusion). 
 
 \item This paper refines the regret analysis for Bernoulli estimation theory in \cite{hazan-kale-stochastic} to show that for any sequence of estimators, not only does \emph{there exist} 
 a probability $p \in [1/4,3/4]$ for which the regret grows at least logarithmically,  but 
 \emph{the set of all such values $p$ has measure at least $1/6$}.  
 This is used 
 to establish a $1/6$ result for opportunistic scheduling: If any particular statistics-unaware algorithm is used, and if nature selects the channel according to a Bernoulli process with parameter $p$ that is independently chosen over the unit interval, then   with probability at 
 least $1/6$ the 
 algorithm will be limited by the $\Omega(\log(T)/T)$ converse bound.  Shouldn't algorithms \emph{always} be limited by this  bound? No.  Imagine a scheduling algorithm that makes an \emph{a-priori guess} $\hat{q} \in [0,1]$ about the true network probability $q$, and then makes decisions that are optimal under the assumption that the guess is exact.  In the ``lucky'' situation when $\hat{q}=q$, this algorithm would perform optimally and would not be 
 limited by the $\Omega(\log(T)/T)$ converse.  Nevertheless, our analysis shows that 
 every algorithm (including algorithms that attempt to make lucky guesses) will fail to 
 beat the $\Omega(\log(T)/T)$ converse with probability at least $1/6$.
  
 \end{enumerate} 

%

\section{Bernoulli estimation}  \label{section:Bernoulli-estimation} 

This section gives preliminaries on  estimating an unknown probability $p \in [0,1]$ from i.i.d. Bernoulli samples $\{W_n\}_{n=1}^{\infty}$ with 
$$ P[W_n=1]=p, \quad P[W_n=0]=1-p$$

\subsection{Estimation functions} \label{section:estimation-functions}

Let $\{W_n^p\}_{n=1}^{\infty}$ be a sequence of i.i.d. Bernoulli random variables with $P[W_n^p=1]=p$ (called a \emph{\emph{Bernoulli-$p$} sequence}).  The value of $p \in [0,1]$ is unknown.  On each time step $n$ we observe the value of $W_n^p$ and then make an estimate of $p$ based on all observations that have been seen so far. 

A general estimation method can be characterized as follows: 
Let $\{\hat{A}_n\}_{n=1}^{\infty}$ be an infinite sequence of
functions such that each function $\hat{A}_n(u,w_1, ..., w_n)$ 
maps a binary-valued sequence $(w_1, ..., w_n) \in \{0,1\}^n$ and a random seed $u \in [0,1)$  to a real 
number in the interval $[0,1]$. That is, for all $n\in \{1, 2, 3, ...\}$ we have 
\begin{equation} \label{eq:estimation-functions}
\hat{A}_n:[0,1) \times \{0,1\}^n \rightarrow [0,1]
\end{equation} 
The $\hat{A}_n$  functions shall be called \emph{estimation functions}.  Let $U$ be a random variable that is uniformly distributed over $[0,1)$ and that is independent of $\{W_n^p\}_{n=1}^{\infty}$.  
For $n \in \{1, 2, \ldots\}$, let $A_n^p$ denote the estimate of $p$ based on observations of $W_1^p, \ldots, W_n^p$: 
\begin{equation}
A_n^p = \hat{A}_n(U, W_1^p, W_2^p, \ldots, W_{n}^p) \quad \forall n \in \{1, 2, 3, \ldots\} \label{eq:Ap}
\end{equation}
The random variable $U$ is used to facilitate possibly randomized decisions.\footnote{A real number has infinite precision.  Thus, in principle, a single random variable $U$ that is uniform over $[0,1)$ is sufficient to generate all desired randomness for a sequence of randomized estimates. Indeed, from $U$ we can create an i.i.d. sequence $\{U_i\}_{i=1}^{\infty}$ of uniformly distributed random variables such that each $U_i$ is a deterministic function of $U$. To do this, first obtain an i.i.d. sequence of digits $\{D_i\}_{i=1}^{\infty}$ by writing $U$ in its decimal expansion $U=\sum_{n=1}^{\infty} D_n10^{-n}$. Then let $f:\mathbb{N}^2\rightarrow\mathbb{N}$ be a bijection and define $U_i = \sum_{j=1}^{\infty} D_{f(i,j)}10^{-j}$ for each $i \in \{1, 2, 3, \ldots\}$.}  The functions $\hat{A}_n$ in \eqref{eq:estimation-functions} are only  assumed to be \emph{probabilistically measurable} so that   $A_n^p$  defined in \eqref{eq:Ap} is a valid random variable for all $n \in \{1, 2, 3, \ldots\}$.

For a given sequence of estimator functions, let 
$A_n^p$ be the estimate at time $n$, as defined by 
\eqref{eq:Ap}.  
For each $n \in \{0, 1, 2, \ldots\}$ define $\mathbb{E}_p[(A_n^p-p)^2]$ 
as the mean square estimation error at time  
$n$. 
The expectation is with respect to the random seed $U$ 
and the random Bernoulli sequence.  The random seed $U$ 
is assumed to be independent of the sequence of Bernoulli 
variables $\{W_n^p\}_{n=1}^{\infty}$. 
Thus, if we condition on $U=u$ for a 
particular $u \in [0,1)$, the conditional 
expectation $\mathbb{E}_p[(A_n^p-p)^2 \given U=u]$
is with respect to the probability measure associated only with 
the random vector $(W_1^p, \ldots, W_{n}^p)$.  
For a given random seed 
$u \in [0,1)$ and for two different probabilities $p,q \in [0,1]$, 
the values $\mathbb{E}_p[(A_n^p-p)^2 \given U=u]$ and $E_q[(A_n^q-q)^2 \given U=u]$ are the mean square errors at time $n$ 
associated with the \emph{same deterministic estimation function $\hat{A}_n$} but assuming a Bernoulli-$p$ process and a Bernoulli-$q$ process, respectively.  The following theorem is due to Hazan and Kale in \cite{hazan-kale-stochastic}.


\begin{thm}  \label{thm:hazan-bernoulli-estimation} (Bernoulli estimation from \cite{hazan-kale-stochastic}) 
Fix any sequence of measurable  
estimation functions $\{\hat{A}_n\}_{n=1}^{\infty}$ 
of the form \eqref{eq:estimation-functions}.  
There is a probability $p \in [1/4, 3/4]$ such that 
$$ \sum_{n=1}^N \mathbb{E}_p[(A_n^p-p)^2] \geq \Omega(\log(N)) \quad \forall N \in \{1, 2, 3, ...\}$$
where $A_n^p$ is defined by \eqref{eq:Ap}.  
\end{thm} 

It is important
to distinguish the result of Theorem \ref{thm:hazan-bernoulli-estimation} from the Cramer-Rao estimation bound  (see, for example, \cite{cover-thomas}).  The Cramer-Rao bound is most conveniently applied to 
\emph{unbiased estimators}. While biased 
versions of the Cramer-Rao bound exist, they require additional structural assumptions, such as knowledge of a (differentiable) 
bias function $b'(p)$ with a derivative that is bounded away from $-1$ so that a term $(1+b'(p))^2$ does not vanish.  Moreover, Cramer-Rao bounds are typically applied to a 
single estimator for time step $n$. In contrast, 
the Hazan and Kale theorem above treats the \emph{sum}
mean square error over a \emph{sequence} of estimators, which is essential for establishing connections to the regret of online scheduling algorithms.  

Using the nested interval techniques of \cite{hazan-kale-stochastic}, 
the asymptotic bound $\Omega(\log(N))$ 
in Theorem \ref{thm:hazan-bernoulli-estimation} can be written as an explicit function $b\log(N)-c$ where $b$ and $c$ are system constants that do not depend on $N$.  Unfortunately, there is a 
minor constant factor error in Lemma 15 of \cite{hazan-kale-stochastic}.  That constant minor factor error does not affect correctness of the  $\Omega(\log(N))$ result established in \cite{hazan-kale-stochastic}.\footnote{It should be emphasized that the techniques in \cite{hazan-kale-stochastic} are novel and deep and are in no way diminished by this constant factor error.}  
For convenience, the minor 
error is identified and fixed in Appendix A. 

\subsection{Positive measure in the unit interval} 

Theorem \ref{thm:hazan-bernoulli-estimation} shows that for any sequence of Bernoulli estimators, there is a probability $p \in [1/4, 3/4]$ under which the estimators perform poorly, in the sense of 
having sum mean square error that grows at least logarithmically in $N$.  The next theorem 
shows that, not only does such a probability $p$ exist,  
the set of all such probabilities $p$ is measurable and has measure at least $1/6$ within the interval $[1/4,3/4]$. It also  generalizes to treat arbitrary 
powers of absolute error. 
For each $\alpha>0$ define: 
$$ V_m(\alpha)= \sum_{n=1}^m (1/n)^{\alpha/2} \quad \forall m \in \{1, 2, 3, \ldots\} $$
Notice that
$$ V_m(\alpha) \geq  \int_1^{m+1} (1/t)^{\alpha/2} dt =  \left\{ \begin{array}{ll}
\log(m+1) &\mbox{ if $\alpha=2$} \\
\frac{(m+1)^{1-\alpha/2}-1}{1-\alpha/2} & \mbox{ if $\alpha \neq 2$} 
\end{array}
\right.$$

\begin{thm} \label{thm:positive-measure}  Fix $\alpha \in (0, 2]$.  Fix any sequence of measurable  
estimation functions $\{\hat{A}_n\}_{n=1}^{\infty}$ 
of the form \eqref{eq:estimation-functions}.  Define $\script{Q} \subseteq [1/4, 3/4]$ 
as the set of all $p \in [1/4, 3/4]$ such that:
\begin{equation} 
\limsup_{m\rightarrow\infty} \left[\frac{1}{V_m(\alpha)}\sum_{n=1}^m \mathbb{E}_p[|A_n^p-p|^{\alpha}]\right] \geq \frac{1}{c^{\alpha}2^{3+2\alpha}} \label{eq:in-particular} 
\end{equation} 
where $c=\sqrt{8/3}$. 
Then the set $\script{Q}$ is Lebesgue measurable and has measure $\mu(\script{Q}) \geq 1/6$.  
Thus, a randomly and uniformly chosen $p \in [1/4, 3/4]$ will satisfy the $\limsup$ inequality \eqref{eq:in-particular} with probability at least $1/3$.   In particular, the inequality \eqref{eq:in-particular} implies: 
\begin{itemize} 
\item If $\alpha=2$ then
$$ \limsup_{m\rightarrow\infty} \left[\frac{1}{\log(m+1)}\sum_{n=1}^m \mathbb{E}_p[|A_n^p-p|^{2}]\right] \geq \frac{3}{2^{10}} $$
\item If $\alpha=1$ then
$$ \limsup_{m\rightarrow\infty} \left[\frac{1}{[(m+1)^{1/2}-1]}\sum_{n=1}^m \mathbb{E}_p[|A_n^p-p|]\right] \geq \frac{1}{2^{4}\sqrt{8/3}} $$ 
\end{itemize} 
\end{thm} 
\begin{proof} 
See Appendix B.\footnote{The conference version of this paper in \cite{neely-converse-opportunistic-infocom2020} stated a weaker version of Theorem 2 that showed $\mu(\script{Q})\geq 1/8$ (based on an earlier version of this technical report). The current technical report proves the stronger result $\mu(\script{Q})\geq 1/6$.
An unfortunate error in the conference version  \cite{neely-converse-opportunistic-infocom2020} stated Theorem 2 as holding for $\alpha>0$, whereas it should be stated to hold for $0<\alpha \leq 2$, which ensures $\lim_{m\rightarrow\infty} V_m(\alpha) = \infty$ and $\lim_{m\rightarrow\infty} V_m(\alpha+1)/V_m(\alpha)=0$ (used in the proof of Theorem 2).  The case $\alpha \in (0, 2]$ is the main case of interest for asymptotic tradeoff analysis:  For $\alpha>2$  the sum regret does not necessarily go to infinity with the number of samples because, for the simple estimator $A_n^p= \frac{1}{n}\sum_{j=1}^n W_j^p$,  we have $\sum_{n=1}^{\infty} \mathbb{E}_p[|A_n^p-p|^{\alpha}]<\infty$.} 
\end{proof} 

Appendix C has details on the tightness of these lower bounds by comparing to a particular 
estimation method  that achieves regret to within a constant factor of these bounds.

\section{The Converse Bound} \label{section:converse}

This section constructs a simple 2-user opportunistic scheduling system with state vectors $S[t]$ described by a single probability parameter $q \in [1/4, 3/4]$.  It produces a converse bound on the utility optimality gap by mapping the problem to a Bernoulli estimation problem and then using Theorem \ref{thm:hazan-bernoulli-estimation} and Theorem \ref{thm:positive-measure}. 

\subsection{The example 2-user system} 

Consider a 2-user wireless system that operates in slotted time $t \in \{0, 1, 2, \ldots\}$. Suppose
the system state is described by a sequence of i.i.d. Bernoulli variables $\{S[t]\}_{t=0}^{\infty}$ with 
$$ P[S[t]=1] = q; \quad P[S[t]=0] = 1-q$$
where $q \in [0,1]$ is an unknown probability. Every slot $t \in \{0, 1, 2, \ldots\}$ the system controller observes $S[t]$ and
chooses a \emph{transmission rate vector} $X[t] = (X_1[t], X_2[t])$ from 
a decision set $\script{D}(S[t])$ given by
\begin{equation} \label{eq:decision-set} 
\script{D}(S[t]) = \left\{ \begin{array}{ll}
\{(1,0)\} &\mbox{ if $S[t]=0$} \\
\{(r, 1-r^2) : r \in [0,1]\}  & \mbox{ if $S[t]=1$} 
\end{array}
\right.
\end{equation} 
In particular, if $S[t]=0$ then the controller has no choice but to allocate $X[t]=(1,0)$, which gives no transmission rate to user 2.  On the other hand, if $S[t]=1$ then the controller is free to allocate $X[t]=(r, 1-r^2)$ for some $r \in [0,1]$, which allows giving a nonzero transmission rate to user 2.   Observe that under any system state and any decision, it holds 
that $0\leq X_1[t] \leq 1$, $0\leq X_2[t]\leq 1$, and $X_2[t]=1-X_1[t]^2$ 
for all $t$. The set of all points $(X_1[t], X_2[t]) \in \script{D}(1)$ available when $S[t]=1$ is shown as the solid curve in Fig. \ref{fig:repcurve}. 

\begin{figure}[htbp]
   \centering
   \includegraphics[height=3in]{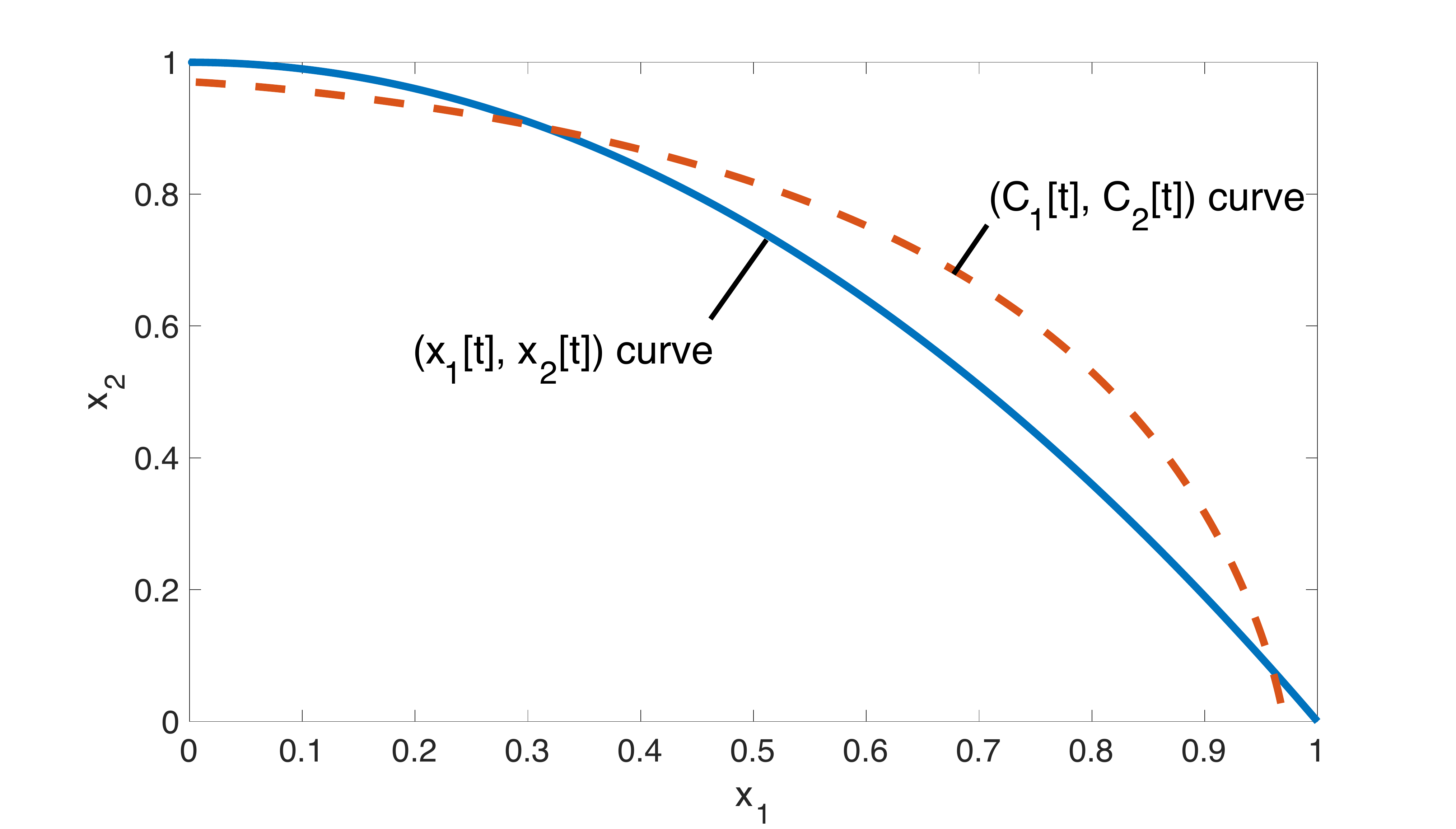} 
   \caption{An illustration of the decision set $\script{D}(1)$ as compared to an alternative logarithmic decision set.  The solid curve shows the
   $(X_1[t], X_2[t])$ points in  $\script{D}(1)$ when $S[t]=1$; the 
   dashed curve shows the points $(C_1[t], C_2[t])$ defined by \eqref{eq:log1}-\eqref{eq:log2}.}
   \label{fig:repcurve}
\end{figure}

While this example decision set $\script{D}(S[t])$ 
is very specific, it is representative of the following physical scenario:
Imagine that user $2$ goes offline independently every slot $t$ with probability $1-q$ (possibly 
due to a time-varying channel condition, or because it allocates its resources to other tasks according to a randomized schedule).  Hence, user 1
can allocate a full rate of 1 on those slots (corresponding to slots $t$ such that $S[t]=0$).  On the other hand, during the slots  in which users 1
and 2 are both online (corresponding to $S[t]=1$), the users can simultaneously transmit but, due to interference, they cannot both transmit at the full rate of 1.  During such slots $t$ for which $S[t]=1$, there is a tradeoff between the rates $X_1[t]$ and $X_2[t]$ that can be allocated, so that $X_2[t]$ is a nonincreasing function of $X_1[t]$.  The particular nondecreasing function $X_2[t] = 1-X_1[t]^2$ that is used is shown in Fig. \ref{fig:repcurve}. This function is chosen for mathematical convenience (it simplifies the proof to be given).   Similar proofs can be given for curves that are qualitatively
similar but that have more physical meaning:  For example, for slots $t$ such that $S[t]=1$, 
suppose the total bandwidth available is $B$ and the rates of users 1 and 2 are chosen by allocating fractions of the bandwidth $\theta_1[t]$ and $\theta_2[t]$ to users 1 and 2, so that user 1 is allocated a total bandwidth of $B\theta_1[t]$, user 2 is allocated a total bandwidth 
of $B\theta_2[t]$, and $\theta_1[t],\theta_2[t]$ are chosen as 
nonnegative values that sum to 1.  The users thus transmit over frequency-separated channels.  Assuming each channel is an additive white Gaussian noise channel (with noise density 
uniform over the given frequency spectrum) 
and given these particular frequency division allocations on slot $t$,  
the point-to-point Shannon capacity of each channel is \cite{cover-thomas}: 
\begin{align}
 C_1[t] &= \theta_1[t] B \log\left(1 + \frac{P}{\theta_1[t] N}\right) \label{eq:log1} \\
 C_2[t] &= \theta_2[t] B \log\left(1 + \frac{P}{\theta_2[t]N}\right)\label{eq:log2} 
 \end{align}
where $P$ and $N$ are fixed positive parameters. The expression $\frac{P}{\theta_i[t]N}$ represents the signal-to-noise ratio for channel $i \in \{1,2\}$ and the noise $\theta_iN$  is proportional to the bandwidth used on channel $i$. 

The $(C_1[t], C_2[t])$ values traced out by all possible $(\theta_1[t], \theta_2[t])$ allocations are given in Fig. \ref{fig:repcurve} for a particular choice of parameters $B=0.7, P/N=3$.  The mathematical curve is different from 
the curve $(r, 1-r^2)$, but it is qualitatively similar. In particular, like the curve $(r, 1-r^2)$, 
it can be shown to have a \emph{strongly concave structure}. 
The proof of our converse can be extended to apply to this particular $(C_1[t], C_2[t])$
curve, and to similar curves that are strongly concave. 
We use the curve $(r,1-r^2)$ because it is simple and yields the most direct proof of the desired converse result.

\subsection{The example network utility maximization problem} 

For positive integers $T$ define:
$$ \overline{X}_i(T) = \frac{1}{T}\sum_{t=0}^{T-1} X_i[t]  \quad \forall i \in \{1,2\}$$
and define $\overline{X}(T) = (\overline{X}_1(T), \overline{X}_2(T))$. 
Let $\phi:[0,1]^2 \rightarrow \mathbb{R}$ be a continuous and 
concave utility function.  The goal of the network controller is to allocate
$X[t]$ over time to solve
\begin{align}
\mbox{Maximize:} \quad & \liminf_{T\rightarrow\infty} \phi\left(\expect{\overline{X}(T)}\right) \label{eq:u1} \\
\mbox{Subject to:} \quad & X[t] \in \script{D}(S[t]) \quad \forall t \in \{0, 1, 2, \ldots\} \label{eq:u2} 
\end{align}
This problem indeed has the structure of \eqref{eq:p1}-\eqref{eq:p2}. 
Define $\phi^*$ as the optimal utility for the above problem. That is, $\phi^*$ is the supremum value of \eqref{eq:u1} over all algorithms that  satisfy \eqref{eq:u2}.   The specific utility function that we first consider is: 
$$ \phi(x_1, x_2) = \log(1+x_1) + \log(1+x_2)$$ 
This function is entrywise nondecreasing and $c$-\emph{strongly concave} over the domain $(x_1, x_2) \in [0,1]^2$ with parameter $c=1/4$, meaning that the function $s:[0,1]^2\rightarrow \mathbb{R}$ given by 
$$s(x_1,x_2) = \phi(x_1,x_2) + \frac{c}{2}(x_1^2+x_2^2)$$
is concave. 
The main result of the paper is below: 

\begin{thm} \label{thm:performance} Consider 
the 2-user example problem \eqref{eq:u1}-\eqref{eq:u2} with utility function $\phi(x_1,x_2)=\log(1+x_1) + \log(1+x_2)$ and with an unknown parameter $q \in [0,1]$.  Under any (possibly randomized) statistics-unaware control algorithm, there is a probability $q \in [1/4, 3/4]$ such that 
\begin{equation} \label{eq:performance1}
\phi\left(\expect{\overline{X}_1[T]}, \expect{\overline{X}_2[T]}\right) \leq \phi^* - \Omega\left(\frac{\log(T)}{T}\right) \quad \forall T \in \{1, 2, 3, \ldots\}
\end{equation} 
Furthermore, there is a measurable subset $\script{Q} \subseteq [1/4, 3/4]$ with measure at least $1/6$ such that if $q \in \script{Q}$ then 
\begin{equation} \label{eq:performance2}
\limsup_{T\rightarrow\infty} \frac{T\phi^*-T\phi\left(\expect{\overline{X}_1[T]}, \expect{\overline{X}_2[T]}\right)}{\log(T)} \geq \frac{3\beta^2}{2^{13}}
\end{equation}
where $\beta=\frac{2}{3}-\frac{\sqrt{7}}{6} \approx 0.2257$. 
\end{thm} 

The proof of Theorem \ref{thm:performance} 
is developed in the following subsections. 
Observe that since the utility function $\phi(x_1,x_2)=\log(1+x_1)+\log(1+x_2)$ is smooth, 
the result in \cite{neely-frank-wolfe-ton} ensures that the statistics-unaware 
Frank-Wolfe algorithm with vanishing stepsize can be used, without 
knowledge of the parameter $q \in [0,1]$, to ensure that for all values $q \in [0,1]$ we have 
$$  \phi\left(\expect{\overline{X}_1[T]}, \expect{\overline{X}_2[T]}\right) \geq \phi^* - O\left(\frac{\log(T)}{T}\right) \quad \forall T \in \{1, 2, 3, \ldots\}$$
In particular, the $\log(T)/T$ asymptotic converse bound of Theorem \ref{thm:performance} can be achieved. Hence, the asymptotic $\log(T)/T$   bound is tight and the corresponding Frank-Wolfe algorithm (with vanishing stepsize) is asymptotically optimal. 

\subsection{Optimality over stationary policies} 

\begin{figure}[htbp]
   \centering
   \includegraphics[height=3in]{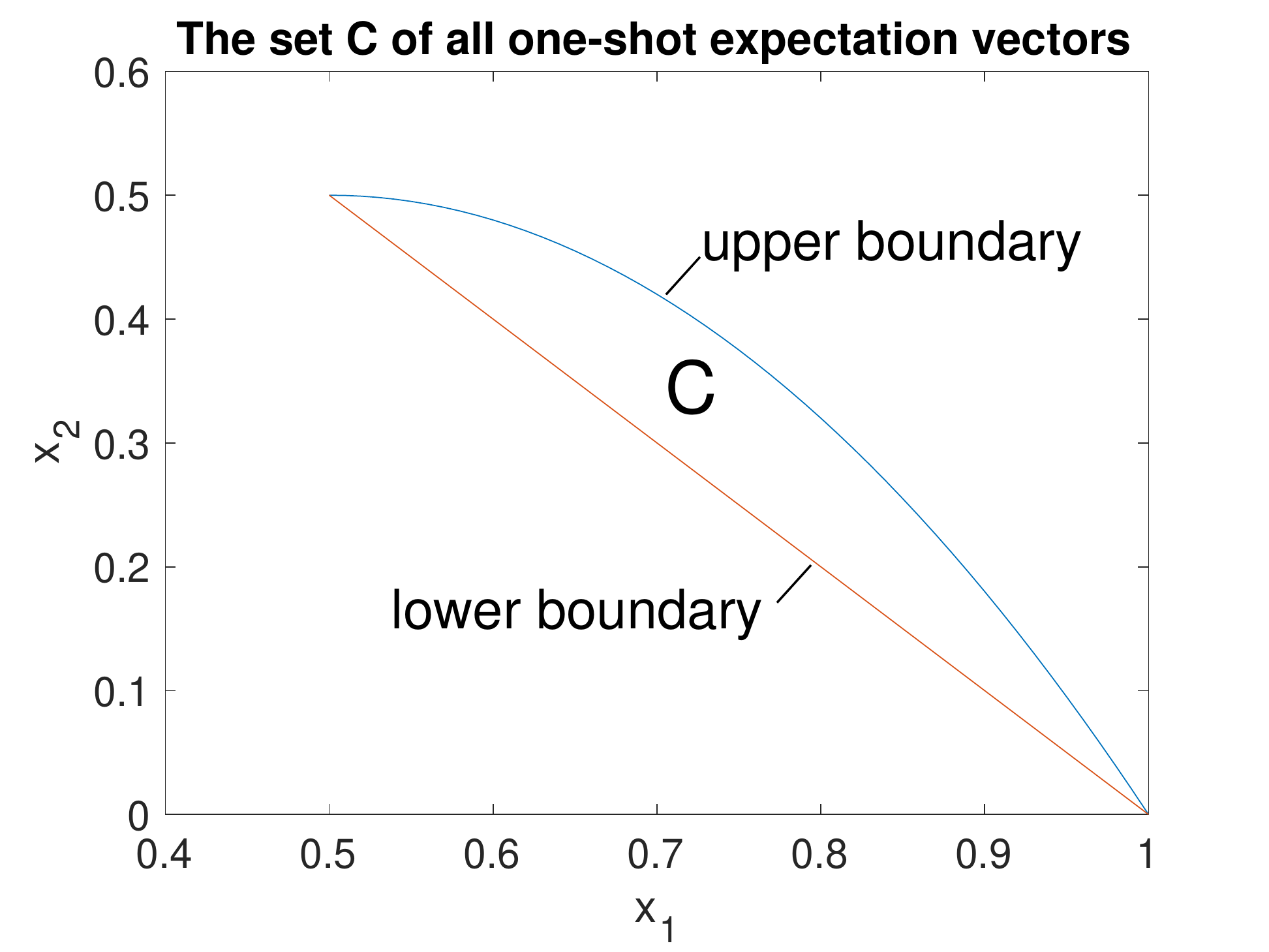} 
   \caption{The set $\script{C}$ of all one-shot expectations for the case $q=0.5$, which includes all points on or between the upper and lower boundary curves.  The lower boundary
   is determined by \eqref{eq:Cb} and the upper boundary is determined by \eqref{eq:Cc}.}
   \label{fig:ITplot}
\end{figure}

Results in \cite{sno-text} show that optimality for the problem \eqref{eq:u1}-\eqref{eq:u2} can be characterized by the set $\script{C}$ of all one-shot expectations $\expect{(X_1[0], X_2[0])}$ that can be achieved on slot $t=0$.  Consider the set of all $(x_1,x_2) \in \mathbb{R}^2$ that satisfy 
\begin{align}
&1-q \leq x_1 \leq 1 \label{eq:Ca}  \\
&x_1 + x_2 \geq 1 \label{eq:Cb} \\
&x_2 \leq 2(1-x_1) - \frac{1}{q}(1-x_1)^2 \label{eq:Cc} 
\end{align} 
The set of all such points is shown in Fig. \ref{fig:ITplot} for the case $q=0.5$. The next result shows that this set is equal to $\script{C}$. 

\begin{lem}  \label{lem:utility-one-shot}  (The set $\script{C}$)  Fix $q \in (0,1]$ and assume $\{S[t]\}_{t=0}^{\infty}$ is an i.i.d. Bernoulli process with $P[S[t]=1]=q$. The set $\script{C}$ of all one-shot expectations $\expect{(X_1[0],X_2[0])}$ achievable in the 2-user system is 
\begin{equation} \label{eq:C1} 
 \script{C} = \Conv(\{(1-q + qr, q(1-r^2))  \in \mathbb{R}^2 : r \in [0,1]\}) 
 \end{equation} 
where $\Conv(\cdot)$ denotes the convex hull.   This set $\script{C}$ is closed, bounded, convex, and is equivalently 
described as the set of all $(x_1, x_2) \in \mathbb{R}^2$ that satisfy 
the inequalities \eqref{eq:Ca}, \eqref{eq:Cb}, \eqref{eq:Cc}. 
\end{lem} 
\begin{proof} 
Define the set $\script{C}_1$ by 
$$\script{C}_1 = \{(1-q + qr, q(1-r^2)) \in \mathbb{R}^2 : r \in [0,1]\}$$
The set $\Conv(\script{C}_1)$ corresponds to the set defined in the right-hand-side of  \eqref{eq:C1}.  The set $\script{C}_1$ is closed and bounded and so $\Conv(\script{C}_1)$ is convex, closed, and bounded. 
Define $\script{C}_2$ as the set of all points $(x_1, x_2)$ that 
satisfy the three inequality constraints \eqref{eq:Ca}, \eqref{eq:Cb}, \eqref{eq:Cc}. 
Define $\script{C}$ as the set of all expectations $\expect{(X_1[0], X_2[0])}$ 
achievable on slot $t=0$.  We want to show that $\Conv(\script{C}_1) = \script{C}=\script{C}_2$.

We first show  $\Conv(\script{C}_1) \subseteq \script{C}$.  
Fix $r \in [0,1]$.  Consider the following decision policy for slot $t=0$: 
$$ (X_1[0], X_2[0]) = \left\{\begin{array}{cc}
(1,0) & \mbox{ if $S[0]=0$} \\
(r, 1-r^2) & \mbox{ if $S[0]=1$}
\end{array}\right.$$ 
Then
\begin{align*}
\expect{(X_1[0],X_2[0])} &= (1,0)P[S[0]=0] + (r, 1-r^2)P[S[0]=1]\\
&= (1,0)(1-q) + (r, 1-r^2)q\\
&= ((1-q) + qr , q(1-r^2))
\end{align*}
Hence, any point in the set $\script{C}_1$ 
can be achieved as an expectation on slot $t=0$. 
Any point in $\Conv(\script{C}_1)$ can be achieved by randomizing over policies that achieve particular points in $\script{C}_1$.  Thus, $\Conv(\script{C}_1) \subseteq \script{C}$. 

We now show that $\script{C} \subseteq \script{C}_2$.   Consider any point $(x_1, x_2) \in \script{C}$.  By definition of $\script{C}$, the point $(x_1, x_2)$ can be achieved as an expectation on slot $t=0$. Let $(X_1[0], X_2[0])\in \script{D}(S[0])$ denote the (possibly randomized) decision vector on slot $0$  that satisfies 
$\expect{(X_1[0], X_2[0])} = (x_1, x_2)$. 
For simplicity of notation define $S=S[0]$ and $(X_1,X_2) = (X_1[0],X_2[0])$ so that: 
\begin{align*}
&(X_1, X_2) \in \script{D}(S)\\
&\expect{(X_1, X_2))} = (x_1, x_2)
\end{align*}
By the structure of the decision set $\script{D}(S)$ we have $X_1 \in [0,1]$ always, and 
\begin{equation} \label{eq:structure} 
(X_1, X_2) = \left\{\begin{array}{cc}
(1,0) & \mbox{ if $S=0$} \\
(X_1, 1-X_1^2) & \mbox{ if $S=1$}
\end{array}\right.
\end{equation} 
Since $P[S=1]=q$ we have 
\begin{align*}
x_1 &= \expect{X_1} \\
&= \expect{X_1|S=0}(1-q) + \expect{X_1|S=1}q\\
&= (1-q) + \expect{X_1|S=1}q 
\end{align*}
Thus
\begin{equation} \label{eq:foo1} 
x_1 = (1-q) + \expect{X_1|S=1}q 
\end{equation} 
Since $X_1 \in [0,1]$ we have $0\leq \expect{X_1|S=1}\leq 1$ and so \eqref{eq:foo1} implies that \eqref{eq:Ca} holds. 

To show that \eqref{eq:Cb} also holds, observe from \eqref{eq:structure} that
regardless of whether $S=0$ or $S=1$
we have:
$$ X_1 + X_2 \geq 1 $$
which holds because $X_1 \in [0,1]$ and $X_1^2\leq X_1$, 
and so 
$$\expect{X_1 + X_2} \geq 1$$
which implies that inequality \eqref{eq:Cb} holds. 

To show that \eqref{eq:Cc} holds, observe that 
\begin{align}
x_2 &= \expect{X_2} \nonumber \\
&\overset{(a)}{=} 0 P[S=0]+ \expect{1-X_1^2|S=1}P[S=1] \nonumber \\
&= q - q \expect{X_1^2|S=1}\nonumber \\
&\overset{(b)}{\leq} q-q\expect{X_1|S=1}^2\nonumber \\
&\overset{(c)}{=} q - q \left(\frac{x_1-1+q}{q}\right)^2 \\
&= 2(1-x_1) - \frac{(1-x_1)^2}{q}
\end{align}
where (a) holds by \eqref{eq:structure}; (b) holds by Jensen's inequality; 
(c) holds by \eqref{eq:foo1}. Thus, inequality \eqref{eq:Cc} holds. 

It follows that $\Conv(\script{C}_1) \subseteq \script{C} \subseteq \script{C}_2$.  Finally, 
it is not difficult to show that the upper boundary of the set $\script{C}_2$, defined by all points $(x_1, x_2)$ that satisfy $1-q\leq x_1\leq 1$ and that satisfy inequality \eqref{eq:Cc} with equality, is equal to the set $\script{C}_1$ (see Fig. \ref{fig:ITplot} for the case $q=0.5$).   Further, the convex hull of this upper boundary is the entire set $\script{C}_2$, so that $\Conv(\script{C}_1) = \script{C}_2$. 
\end{proof} 

Results in \cite{sno-text} imply that the optimal utility $\phi^*$ for problem \eqref{eq:u1}-\eqref{eq:u2} is equal to the
maximum of the continuous and concave function 
$\phi(x_1,x_2)$ over all $(x_1,x_2)$ in the closed, bounded, and convex set $\script{C}$. That is, 
$$ \phi^* = \max_{(x_1,x_2) \in \script{C}} \phi(x_1,x_2)$$
The utility function of interest can be shown to be strongly concave and so the maximizer $(x_1^*, x_2^*)$ is unique.  The maximizer  is given in terms of $q$
in the next lemma. 

\begin{lem} \label{lem:derivative-properties} (Optimal operating point) Fix $q \in (0,1]$.  Define 
$\phi:[0,1]^2\rightarrow\mathbb{R}$ 
by $\phi(x_1,x_2) = \log(1+x_1)+\log(1+x_2)$. The unique maximizer of $\phi(x_1,x_2)$ over $(x_1,x_2) \in \script{C}$ is a vector $(x_1^*,x_2^*)$ that satisfies 
\begin{equation} \label{eq:x-star} 
(x_1^*, x_2^*) = \left((1-q) + qr, q(1-r^2)\right)
\end{equation} 
for some particular $r \in [0,1]$, so that $\phi(x_1^*,x_2^*)=\phi^*$.  Furthermore the optimal value $r \in [0,1]$ in \eqref{eq:x-star} 
satisfies 
\begin{equation} \label{eq:derivative-equality}
\frac{1}{1+ x_1^*} = \frac{2r}{1+x_2^*}
\end{equation} 
and is exactly equal to 
\begin{equation} \label{eq:r-expression}
r = \frac{-(2-q) + \sqrt{4q^2-q+4}}{3q} 
\end{equation} 
The expression on the right-hand-side of \eqref{eq:r-expression} 
increases from $1/4$ to  $\frac{-1+\sqrt{7}}{3}$ as $q$ slides between $0$ and $1$, 
where $\frac{-1+\sqrt{7}}{3} \approx 0.54858$. 
\end{lem} 
\begin{proof} 
Considering the set $\script{C}$ defined by \eqref{eq:C1}, it is clear that (because $\phi(x_1,x_2)$ is entrywise nondecreasing) a maximizer occurs on the 
upper boundary curve $(1-q+qr, q(1-r^2))$ for $r \in [0,1]$.  The utility function associated with $r \in [0,1]$ is 
$$ \log(2 - q + qr)  + \log(1 + q(1-r^2))$$
This is concave over $r \in [0,1]$. If a point of zero derivative can be found over $r \in [0,1]$ then that point must be optimal.  Taking a derivative with respect to $r$ and setting the result to $0$ yields
$$ \frac{q}{2-q+qr} + \frac{-2qr}{1+q(1-r^2)} =0$$
Since $q>0$, dividing by $q>0$ and rearranging terms gives
$$ \frac{1}{2-q+qr} = \frac{2r}{1+q(1-r^2)}$$
which yields \eqref{eq:derivative-equality} by the substitution $x_1^*=1-q+qr$, $x_2^*=q(1-r^2)$.  Rearranging the above equality again yields a quadratic equation in $r$ that is solved by taking the only nonnegative solution, which is given in \eqref{eq:r-expression}.  It can be checked that the expression in \eqref{eq:r-expression} increases from $1/4$ to  $\frac{-1+\sqrt{7}}{3}$ as $q$ slides between $0$ and $1$. In particular, a zero-derivative point $r \in [0,1]$ can indeed be found for all values of $q$ that are considered. 
\end{proof} 

\begin{lem} \label{lem:h-def} (The bijection $h$) Define the function $h:[0,1]\rightarrow[1/4, \frac{-1+\sqrt{7}}{3}]$ by 
\begin{equation} \label{eq:h} 
h(q) = \left\{ \begin{array}{ll}
1/4 &\mbox{ if $q=0$} \\
  \frac{-(2-q) + \sqrt{4q^2-q+4}}{3q}& \mbox{ otherwise} 
\end{array}
\right.
\end{equation} 
Then $q \in (0,1]$ implies $r=h(q)$ (where $r$ is defined in \eqref{eq:r-expression}). Further, 
 function $h$ is strictly  increasing, 
 so that it has an inverse function: 
 $$h^{-1}:\left[1/4, \frac{-1+\sqrt{7}}{3}\right]\rightarrow [0,1]$$ 
Finally, $h$ is continuously 
differentiable  and satisfies $h'(q)\geq h'(1)$ for all $q \in [0,1]$, 
and so defining $\beta>0$ by 
\begin{equation} \label{eq:beta} 
\beta = h'(1) = \frac{2}{3}-\frac{\sqrt{7}}{6} \approx 0.2257
\end{equation} 
we have the following  ``expansion'' property of $h$: 
\begin{equation} \label{eq:expansion} 
 |h(a)-h(b)|\geq \beta |a-b| \quad \forall a, b \in [0,1]
 \end{equation} 
\end{lem} 
\begin{proof} 
Taking a derivative of $h$ at $q \in (0,1]$ gives
$$ h'(q) = \frac{q-8+4\sqrt{4q^2-q+4}}{6q^2\sqrt{4q^2-q+4}}$$
This is continuous at all $q \in (0,1]$.  Further, 
$$ \lim_{q\rightarrow 0^+} h'(q) = \frac{21}{64}$$
This is consistent with the right-derivative of $h$ at $0$: 
$$ \lim_{q\rightarrow 0^+} \frac{h(q)-1/4}{q} = \frac{21}{64}$$
Thus, $h$ is continuously differentiable for all $q \in [0,1]$. 
From the expression for $h'(q)$ it follows that $h'(1) = \beta = \frac{2}{3}-\frac{\sqrt{7}}{6}$. 
As $q$ slides over the interval $[0,1]$, the function $h'(q)$ at first increases but eventually decreases
to reach a minimum value at $q=1$. Thus,   $h'(q) \geq  \beta$ for all $q \in [0,1]$. 
The expansion property \eqref{eq:expansion}  
follows by fixing $a,b \in [0,1]$, without loss of generality assuming $a<b$, and  using the fundamental theorem of calculus: 
\begin{align*}
h(b) -h(a) &=  \int_a^b h'(q)dq \\
&\geq  \int_a^b \beta dq \\
&=  \beta(b-a)
\end{align*}
where the inequality holds because $h'(q)\geq \beta$ for all $q \in [0,1]$.
\end{proof} 

\subsection{Statistics-unaware algorithms for utility maximization} 

This subsection completes the proof of Theorem \ref{thm:performance} for the 2-user 
problem \eqref{eq:u1}-\eqref{eq:u2}. 
Fix $q \in [1/4,3/4]$ and recall that $q=P[S[t]=1]$.  
Consider any \emph{statistics-unaware} 
algorithm for choosing $X[t]=(X_1[t],X_2[t]) \in \script{D}(S[t])$ over $t \in \{0, 1, 2, \ldots\}$, where
$\script{D}(S[t])$ is given in \eqref{eq:decision-set}.  In particular, the 
algorithm has no a-priori knowledge of $q$.   For each $t \in \{0,1, 2, \ldots\}$ define
$$x[t]= (x_1[t],x_2[t]) = \expect{(X_1[t], X_2[t])}$$
Fix $T$ as a positive integer.  Define $\overline{X}[T]$ as
the time average over the first $T$ slots 
\begin{equation*} 
\mbox{$\overline{X}[T]=\frac{1}{T}\sum_{t=0}^{T-1}(X_1[t], X_2[t])$}
\end{equation*} 
Taking expectations of both sides of the above equality and using the definition of $x[t]$ gives 
\begin{equation} \label{eq:so-that} 
\mbox{$\expect{\overline{X}[T]} = \frac{1}{T}\sum_{t=0}^{T-1} x[t]$}
\end{equation} 
Let $x^*=(x_1^*,x_2^*)$ be the optimal operating point  defined in \eqref{eq:x-star} of Lemma \ref{lem:derivative-properties}.  Let $\phi'(x^*)^{\transpose}$ denote the gradient at $x^*$ expressed as a row vector: 
\begin{align}
&\phi(x_1,x_2) = \log(1+x_1) + \log(1+x_2) \nonumber \\
&\implies  \phi'(x^*)^{\transpose} = \left[\frac{\partial \phi(x^*)}{\partial x_1}; \frac{\partial \phi(x^*)}{\partial x_2}\right] = \left[\frac{1}{1+x_1^*}; \frac{1}{1+x_2^*}  \right] \label{eq:gradient} 
\end{align}
By concavity of $\phi$  we have:
\begin{align}
 \phi(\expect{\overline{X}(T)}) &\overset{(a)}{\leq} \phi(x^*) + \phi'(x^*)^{\transpose}(\expect{\overline{X}(T)}-x^*)  \nonumber \\
 &\overset{(b)}{=} \phi(x^*) +  \phi'(x^*)^{\transpose} \cdot \frac{1}{T} \sum_{t=0}^{T-1}(x[t]-x^*) \nonumber\\
 &\overset{(c)}{=}\phi^* + \frac{1}{T}\sum_{t=0}^{T-1} \left[\frac{x_1[t]-x_1^*}{1+x_1^*} + \frac{x_2[t]-x_2^*}{1+x_2^*}\right] \nonumber \\
 &\overset{(d)}{=}\phi^* + \frac{1}{T}\sum_{t=0}^{T-1}\left[\frac{2r(x_1[t]-x_1^*)}{1+x_2^*} + \frac{(x_2[t]-x_2^*)}{(1+x_2^*)} \right]  \label{eq:subgradient-inequality2}
 \end{align}
 where (a) holds by the gradient inequality for concave functions; (b) 
 holds by   \eqref{eq:so-that}; (c) 
 holds by \eqref{eq:gradient} and the fact that $\phi(x^*)=\phi^*$ 
 for the vector $x^*=(x_1^*,x_2^*)$ defined in Lemma \ref{lem:derivative-properties}; 
 (d)  holds by \eqref{eq:derivative-equality}.  
 
 Now consider a particular $t \in \{0, 1, \ldots, T-1\}$.  Define $F[t]$ as the history of channel states over the first $t$ slots: 
 $$ F[t] = (S[0], S[1], \ldots, S[t-1])$$
 with $S[-1]$ defined as a nonrandom constant so that $F[0]$ provides no information about the channel.  Define $z_{F[t]}$ by 
 \begin{equation} \label{eq:ZF} 
z_{F[t]} = \expect{X_1[t]|S[t]=1, F[t]}
 \end{equation} 
 By Jensen's inequality we have 
 \begin{equation} \label{eq:ZF2} 
 z_{F[t]}^2 \leq \expect{X_1[t]^2|S[t]=1, F[t]} 
 \end{equation} 
 We have 
 \begin{align}
 &\expect{(X_1[t], X_2[t])|F[t]}\nonumber \\
 &\overset{(a)}{=}\expect{(X_1[t],X_2[t])|S[t]=0,F[t]}(1-q)\nonumber \\
 & \quad + \expect{(X_1[t],X_2[t])|S[t]=1, F[t]}q\nonumber \\
 &\overset{(b)}{=}(1,0)(1-q) + \expect{(X_1[t],1-X_1[t]^2)|S[t]=1, F[t]}q \nonumber \\
 &\overset{(c)}\leq (1,0)(1-q) + (z_{F[t]}, 1-z_{F[t]}^2)q \nonumber \\
 &= (1-q + qz_{F[t]}, q(1-z_{F[t]}^2)) \label{eq:other-hand1} 
 \end{align}
 where (a) holds by conditioning on the events $\{S[t]=1\}$ and $\{S[t]=0\}$ (which are independent of $F[t]$); (b) holds by the decision set structure 
  $\script{D}(S[t])$ in \eqref{eq:decision-set}; (c) holds as an entrywise inequality  by \eqref{eq:ZF} and \eqref{eq:ZF2}. Taking expectations of \eqref{eq:other-hand1} with respect to the random $F[t]$ and using the law of iterated expectations gives (using $\expect{X_i[t]} = x_i[t]$): 
\begin{equation} \label{eq:other-hand-3} 
(x_1[t], x_2[t]) \leq \left(1-q + q\expect{z_{F[t]}}, q-q\expect{z_{F[t]}^2}\right) 
\end{equation} 
  
   On the other
 hand, recall from Lemma \ref{lem:derivative-properties} that 
\begin{equation} \label{eq:other-hand2}
(x_1^*,x_2^*) = (1-q + qr, q(1-r^2)) 
\end{equation}
Using \eqref{eq:other-hand-3} and \eqref{eq:other-hand2} together gives 
\begin{align}
&2r(x_1[t]-x_1^*) + (x_2[t]-x_2^*) \nonumber \\
&\leq  2rq(\expect{z_{F[t]}}-r) + q(r^2-\expect{z_{F[t]}^2}) \nonumber \\
&= -q\expect{(z_{F[t]}-r)^2}\nonumber \\
&\leq -q\expect{\left([z_{F[t]}]_{h(0)}^{h(1)}-r\right)^2} \label{eq:star} 
\end{align}
where $[z_{F[t]}]_{h(0)}^{h(1)}$ projects $z_{F[t]}$ onto the interval $[h(0),h(1)]$ and the final inequality holds because 
we know $r \in [h(0),h(1)]$, and the distance between $z_{F[t]}$ and $r$ must be greater than or equal to their distances when projected onto the interval $[h(0),h(1)]$.   Now we know that $h:[0,1]\rightarrow [h(0),h(1)]$ is bijective and so we can define
\begin{equation} \label{eq:theta-def} 
\theta[t] = h^{-1}\left([z_{F[t]}]_{h(0)}^{h(1)}\right)
\end{equation} 
so that $[z_{F[t]}]_{h(0)}^{h(1)}=h(\theta[t])$.  Substituting this and $r=h(q)$ in \eqref{eq:star} gives 
$$2r(x_1[t]-x_1^*) + (x_2[t]-x_2^*) \leq -q\expect{\left(h(\theta[t])-h(q)\right)^2}$$
Substituting the above inequality into \eqref{eq:subgradient-inequality2}  yields: 
\begin{align} 
\phi(\expect{\overline{X}(T)}) &\leq \phi^* - \frac{q}{T(1+x_2^*)}\sum_{t=0}^{T-1} \expect{\left(h(\theta[t])-h(q)\right)^2}\nonumber \\
&\overset{(a)}{\leq} \phi^*- \frac{1}{8T}\sum_{t=0}^{T-1}\expect{(h(\theta[t])-h(q))^2}\nonumber\\
&\overset{(b)}{\leq} \phi^* - \frac{1}{8T}\sum_{t=1}^{T-1}\expect{(h(\theta[t])-h(q))^2} \label{eq:pre-expansion} 
\end{align}
where (a) holds by
the fact that  $x_2^* \in [0,1]$ and $q \in [1/4, 3/4]$ so that $q/(1+x_2^*) \geq 1/8$; (b) holds by neglecting the nonnegative term for $t=0$. By the expansion property of $h$ in \eqref{eq:expansion} we obtain 
\begin{equation} \label{eq:box}
\boxed{\phi(\expect{\overline{X}(T)}) \leq \phi^*- \frac{\beta^2}{8T}\sum_{t=1}^{T-1}\expect{(\theta[t]-q)^2}}
\end{equation} 
where $\beta=h'(1)\approx 0.2257$ is defined in Lemma \ref{lem:h-def}. 

\

{\bf Here is the crucial observation:}  For each slot $t \in \{1, 2, 3, \ldots\}$ we can view 
$\theta[t]$ as defined in \eqref{eq:theta-def}  
as a deterministic estimator of $q$ based on the $t$ observations $\{S[0], S[1], \ldots, S[t-1]\}$.

\

Indeed, $z_{F[t]}$ defined in \eqref{eq:ZF} is a deterministic function of $F[t]=(S[0], \ldots, S[t-1])$, 
and $\theta[t]$ as defined in \eqref{eq:theta-def} is determined by first projecting $z_{F[t]}$ 
to the interval $[h(0),h(1)]$ and then mapping the result through the continuous deterministic function 
$h^{-1}:[h(0),h(1)]\rightarrow [0,1]$.  In particular, $\theta[t] \in [0,1]$. With this observation, we get from Theorem \ref{thm:hazan-bernoulli-estimation}  that there exists a value $q \in [1/4, 3/4]$ such that: 
$$ \sum_{t=1}^{T-1} \expect{(\theta[t] - q)^2} \geq \Omega(\log(T)) \quad \forall T \in \{1, 2, 3, \ldots\} $$
Substituting this into \eqref{eq:box}  
gives
$$ \phi(\expect{\overline{X}[T]})  \leq \phi^*-  \Omega(\log(T)/T)$$
This establishes the inequality \eqref{eq:performance1} of Theorem \ref{thm:performance}. 

Finally, assuming $T \geq 2$ and rearranging \eqref{eq:box} gives 
\begin{equation}\label{eq:finally}
\frac{T\phi^*-T\phi(\expect{\overline{X}[T]})}{\log(T)} \geq \frac{\beta^2}{8\log(T)}\sum_{t=1}^{T-1}\expect{(\theta[t]-q)^2}
\end{equation} 
Again observing that $\{\theta[t]\}_{t=1}^{\infty}$ is a sequence of deterministic estimators of $q$, from Theorem \ref{thm:positive-measure} we know there is a set $\script{Q} \subseteq [1/4, 3/4]$ with measure $\mu(\script{Q})\geq 1/6$ such that for all $q \in \script{Q}$ we have 
$$ \limsup_{T\rightarrow\infty} \frac{1}{\log(T)} \sum_{t=1}^{T-1} \expect{(\theta[t]-q)^2} \geq \frac{3}{2^{10}}$$
Substituting this into \eqref{eq:finally} yields
$$\limsup_{T\rightarrow\infty} \frac{T\phi^*-T\phi(\expect{\overline{X}[T]})}{\log(T)} \geq \frac{\beta^2}{8}\frac{3}{2^{10}}$$
This completes the proof of Theorem \ref{thm:performance}. $\Box$

\subsection{Discussion}

The $O(\log(T)/T)$ achievability result derived in \cite{neely-frank-wolfe-ton} 
holds for smooth and concave 
utility functions and does not require strong concavity. The $\Omega(\log(T)/T)$ 
converse bound  of Theorem \ref{thm:performance}  was carried out using a smooth and \emph{strongly 
concave} utility function. This was intentional:  \emph{This shows that, for these opportunistic scheduling problems, 
strong concavity cannot improve the fundamental convergence time.}  This is surprising because strong convexity/concavity is known to significantly improve convergence time in other
optimization scenarios, including deterministic subgradient minimization \cite{nesterov-book}
and online convex programming \cite{online-convex}\cite{kale-universal}. 


\subsection{Extension to other utility functions} 

Replace the utility function $\log(1+x_1)+ \log(1+x_2)$ with the more general 
function  $\phi:[0,1]^2\rightarrow\mathbb{R}$: 
$$ \phi(x_1,x_2) = \phi_1(x_1) + \phi_2(x_2)$$
where $\phi_1(x)$ and $\phi_2(x)$ are concave and strictly increasing over $[0,1]$. 
The converse proof can be repeated with mild additional 
assumptions on $\phi_1$  and $\phi_2$.   The main idea is to use the \emph{implicit function theorem} of real analysis to show existence of a strictly increasing and 
continuously differentiable function $h:(0,1)\rightarrow(0,1)$ (different from the $h$ function given for the log utility function 
in \eqref{eq:h}) such that for each $q \in (0,1)$, the value 
$h(q)$ is the $r$ value needed to define the  optimal operating point $(x_1^*, x_2^*) \in \script{C}$ associated with this new utility: 
$$ (x_1^*, x_2^*) = ((1-q) + qr, q(1-r^2))$$
It must also be shown that there is a $\beta>0$  such that $h'(q)\geq \beta$ for all $q \in [1/4, 3/4]$ so that the proof can proceed from \eqref{eq:pre-expansion} to \eqref{eq:box}. These properties are established in the next lemma. They allow the $\Omega(\log(T)/T)$  
converse proof to be repeated 
using the modified estimator: 
$$\theta[t] = h^{-1}\left([\expect{X_1[t]|S[t]=1}]_{h(1/4)}^{h(3/4)}\right)$$
\begin{lem} \label{lem:general-utility} (General utilities) 
Suppose $\phi_1(x)$ and $\phi_2(x)$ are twice continuously differentiable functions  that satisfy:
\begin{itemize} 
\item Assumption 1: $\phi_1'(x)>0$ and $\phi_2'(x)>0$ for all $x \in [0,1]$.
\item Assumption 2: $\phi_1''(x)<0$ and $\phi_2''(x)<0$ for all $x \in [0,1]$.
\item Assumption 3: $\phi_1'(1) < 2\phi_2'(0)$.\footnote{Note that Assumption 3 is implied by Assumption 2 in the special case $\phi_1(x)=\phi_2(x)$ for all $x \in [0,1]$.}
\end{itemize} 
Then for each $q \in (0,1)$ the equation:
\begin{equation} \label{eq:general-g-equation} 
 \phi_1'(1-q+qr) + \phi_2'(q(1-r^2))(-2r) = 0 
\end{equation} 
has a unique solution $r \in (0,1)$.  Further, there is a continuously differentiable function $h:(0,1)\rightarrow (0,1)$ with this property:  $(q,r)$ satisfies \eqref{eq:general-g-equation} if and only if $r=h(q)$, 
and there is a $\beta>0$ such that $h'(q)\geq \beta$ for all $q \in [1/4, 3/4]$. 
\end{lem} 

\begin{proof} 
Define $g(q,r)$ for $q \in (0,1)$ and $r \in [0,1]$ by 
\begin{equation} \label{eq:gqr}
g(q,r)= \phi_1'(1-q+qr) + \phi_2'(q(1-r^2))(-2r) 
\end{equation} 
For $r \in (0,1)$, we see that $g(q,r)=0$ if and only if \eqref{eq:general-g-equation} holds.   Fix $q \in (0,1)$.  
Since $\phi_1$ and $\phi_2$ are twice differentiable,  
$g(q,r)$ is a continuous function of $r$. We have 
\begin{align*}
g(q,0) &\overset{(a)}{=} \phi_1'(1-q) > 0\\
g(q,1) &\overset{(b)}{=} \phi_1'(1) -2\phi_2'(0) <0
\end{align*}
where (a) holds by Assumption 1 and (b) by Assumption 3.   By the intermediate value theorem,  there must exist a value $r \in (0,1)$ such that $g(q,r)=0$. To show \emph{uniqueness} of this value $r \in (0,1)$, it suffices to show that $g(q,r)$ is strictly decreasing in $r$:
\begin{equation} \label{eq:partial}
\frac{\partial g(q,r)}{\partial r} = \underbrace{\phi_1''(1-q+qr)q}_{<0} + \underbrace{\phi_2''(q(1-r^2))q(2r)^2}_{<0} + \underbrace{(-2)\phi_2'(q(1-r^2))}_{<0} < 0
\end{equation} 
where the underbrace inequalities hold by Assumptions 1, 2, and the inequalities  $0<q<1$, $0<r<1$. 
Thus, uniqueness holds, and for each $q \in (0,1)$ we can define $h(q)$ as the unique value in $(0,1)$ 
such that $g(q,h(q))=0$. 
 Since $g(q,r)$ is continuously differentiable (with respect to both $q$ and $r$) over
the open set $(q,r) \in (0,1) \times (0,1)$,  and since $\partial g/\partial r \neq 0$,  the \emph{implicit function theorem} of real analysis can be applied to conclude the function $h(q)$ is continuously differentiable. 

It remains to prove existence of $\beta>0$ such that that $h'(q) \geq \beta$ for all $q \in [1/4, 3/4]$.
We have $g(q, h(q))=0$ for all $q \in (0,1)$ and so by \eqref{eq:gqr}: 
$$ \phi_1'(1-q+qh(q)) -2h(q) \phi_2'(q(1-h(q)^2))=0 \quad \forall q \in (0,1)$$
Taking a derivative with respect to $q$ gives
\begin{align} 
0&=\phi_1''(1-q+qh(q))(-1+h(q) + qh'(q)) \nonumber \\
& \quad -2h(q)\phi_2''(q(1-h(q)^2))(1 - h(q)^2 - 2qh(q)h'(q)) \nonumber \\
&\quad  -2h'(q) \phi_2'(q(1-h(q)^2)) \label{eq:take-limit} 
\end{align}
Suppose there is a $q \in (0,1)$ for which $h'(q)\leq 0$.  By Assumption 1 we know 
$$2h'(q)\phi_2'(q(1-h(q))^2) \leq  0$$ and so by \eqref{eq:take-limit}: 
\begin{align*}
0  &\geq \underbrace{\phi_1''(1-q+qh(q))}_{<0}\underbrace{(-1+h(q)+qh'(q))}_{<0} \\
&\underbrace{- 2h(q)\phi_2''(1(1-h(q)^2))}_{>0}\underbrace{(1-h(q)^2-2qh(q)h'(q))}_{>0}
\end{align*}
So $0>0$, a contradiction. 
Thus, $h'(q)$ is a  continuous and always positive function over $(0,1)$. It must have a strictly positive minimum value over the compact interval $[1/4,3/4]$, call this minimum $\beta$.
\end{proof} 

The discussion in Section \ref{section:linear} shows that the $\Omega(\log(T)/T)$ converse does not hold for \emph{linear} utility functions.  For intuition on how the proof of Lemma \ref{lem:general-utility} 
would fail with linear utilities, it is easy to see that if $\phi_1(x)=a_1x$, $\phi_2(x)=a_2x$ for some real numbers $a_1>0, a_2>0$, then the $r$ value that solves \eqref{eq:general-g-equation} does not depend on $q$ and hence $h'(q)=0$ for all $q$, so there is no $\beta>0$ such that $h'(q)\geq \beta$.
Assumption 2 of Lemma \ref{lem:general-utility} enforces nonlinearity. Assumption 2 implies 
that the $\phi$ function is strongly concave over the domain $[0,1]^2$.  

\section{Conclusion} 

This paper establishes a converse bound of $\Omega(\log(T)/T)$ on the utility gap for 
opportunistic scheduling. This matches a recently established achievability bound 
of $O(\log(T)/T)$.
This means that $\log(T)/T$ is the optimal asymptotic behavior. 
The bound in this paper was 
proven for an example 2-user system with a strongly concave utility function.
This demonstrates the surprising the result that strong concavity of the utility function cannot improve the asymptotic convergence time for opportunistic scheduling systems. This is in contrast to other optimization  scenarios, such as online convex optimization, where strong convexity/concavity is known to significantly improve asymptotic convergence.  The converse proof constructed a nontrivial mapping of the opportunistic scheduling problem to a Bernoulli estimation problem and used a prior result on the regret associated with Bernoulli estimation.  The paper also develops a refinement on Bernoulli estimation to show that for any sequence of Bernoulli estimators, 
not only do probabilities exist for which the estimators perform poorly, but such probabilities have measure at least $1/6$ in the unit interval. This is used to show that for any opportunistic scheduling algorithm, if nature chooses a Bernoulli state  distribution by selecting the Bernoulli probability uniformly over the unit interval, the algorithm is limited by the $\Omega(\log(T)/T)$ bound 
with probability at least $1/6$. 

The converse bound of this paper was established for a simple 2-user system. This means that \emph{there exist} systems that are limited by the $\Omega(\log(T)/T)$ bound. 
The  techniques in this paper  link opportunistic scheduling to estimation problems and can 
likely be used in future work to investigate bounds on more general  networks, including networks with state variables $S[t]$ that are described by more complex distributions.  This motivates the following open questions: Can refined bounds be established for non-Bernoulli $S[t]$ processes?  Can more detailed coefficients of the $\log(T)/T$ curve be obtained in terms of simple parameters of the distribution on $S[t]$?  The Cramer-Rao bound of estimation theory allows bounds for non-Bernoulli variables that depend on 
\emph{Fisher information} of the underlying probability distribution. 
However, it is currently unclear how to  reduce a general opportunistic scheduling problem to a generalized (non-Bernoulli) 
estimation problem, and it is not clear how to incorporate 
Fisher information concepts to provide ``regret'' type bounds for networks. 

\section*{Appendix A --- A Refined Variation Inequality} 

This appendix refines a Lemma in \cite{hazan-kale-stochastic}
about the total variation distance 
associated with the measure of i.i.d. Bernoulli random 
variables.\footnote{Lemma 15 of \cite{hazan-kale-stochastic} contains a minor
constant factor error: It claims 
$p\log(p/q) + (1-p)\log((1-p)/(1-q)) \leq \epsilon^2/2$ for $q=p+\epsilon$ (a
counter-example is $\epsilon = 1/16$, $p=1/2$). This affects the statement
of Lemma 15 in \cite{hazan-kale-stochastic} 
by a constant factor that does not change the methodology or asymptotic
$\Omega(\log(N))$ results proven in that paper.}   
Let $\Omega$ be a finite and nonempty 
sample space and consider the sigma algebra of all subsets of $\Omega$. 
Let $P[A]$ and $P'[A]$ be two probability measures defined for all  
$A \subseteq \Omega$.  Define the \emph{total variation distance} $v(P,P')$ by
$$ v(P,P') = \sup_{A \subseteq \Omega} |P[A] - P'[A]|$$
The Pinsker inequality states
\begin{equation} \label{eq:pinsker}  
v(P,P') \leq \min\left[\sqrt{\frac{1}{2} D_{KL}(P || P')}, \sqrt{\frac{1}{2} D_{KL}(P' || P)}\right] 
\end{equation} 
where $D_{KL}(P||P')$ is the Kullback-Leibler divergence (in nats) between $P$ and $P'$: 
$$ D_{KL}(P||P') = \sum_{x \in \Omega} P[\{x\}] \log\left(\frac{P[\{x\}]}{P'[\{x\}]}\right)$$

For a fixed positive integer $n$, define the sample space
$\{0,1\}^n$.  We view all $2^n$ outcomes $(x_1, ..., x_n) \in \{0,1\}^n$ 
as possible realizations of 
i.i.d. Bernoulli random variables $(X_1, ..., X_n)$.  For each $p \in [0,1]$, 
define $B_n^p$ as the probability measure on $\{0,1\}^n$  associated with i.i.d. Bernoulli random variables with $P[X_i=1]=p, P[X_i=0]=1-p$.  That is, 
$$ B_n^p[A] = \sum_{(x_1, ..., x_n) \in A} \left[\prod_{i=1}^n p^{x_i}(1-p)^{1-x_i}\right] \quad \forall A \subseteq \{0,1\}^n$$
The next lemma is a refinement of  Lemma 15 in \cite{hazan-kale-stochastic} that fixes a constant factor error and also removes a restriction on $|p-q|$ that was assumed there. 

\begin{lem} \label{lem:c-lemma} Fix $p,q \in [1/4, 3/4]$. Then
$$ v(B_p^n, B_q^n) \leq c|p-q|\sqrt{n}$$
where $c= \sqrt{8/3} \approx 1.6329932$. 
\end{lem} 
\begin{proof} 
Without loss of generality assume $p>q$ and define $\epsilon = p-q$. The 
Pinsker inequality \eqref{eq:pinsker} 
gives
\begin{equation} \label{eq:Pin1} 
 v(P,P') \leq \sqrt{\frac{1}{2}D_{KL}(B_n^p||B_n^q)}
 \end{equation} 
By basic properties of the measure of i.i.d. Bernoulli random variables it holds that
$$ D_{KL}(B_n^p||B_n^q) = n\left[p\log\left(\frac{p}{q}\right) + (1-p)\log\left(\frac{1-p}{1-q}\right)  \right]$$
To compute the right-hand-side of the above equality, we have
\begin{align*}
p\log(\frac{p}{q}) + (1-p)\log(\frac{1-p}{1-q}) &= p\log(1 + \frac{\epsilon}{q}) + (1-p)\log(1-\frac{\epsilon}{1-q}) \\
&\overset{(a)}{\leq} p\left(\frac{\epsilon}{q}\right) + (1-p)\left(\frac{-\epsilon}{1-q}\right)\\
&\overset{(b)}{=}\epsilon^2\left[\frac{1}{q} + \frac{1}{1-q}\right]\\
&\leq \epsilon^2 \sup_{q \in [1/4, 3/4]}\left[\frac{1}{q} + \frac{1}{1-q}\right]\\
&= \frac{16\epsilon^2}{3} 
\end{align*}
where (a) uses the inequality $\log(1+x)\leq x$ for all $x>-1$; (b) uses $p=q+\epsilon$. 
Hence
$$ D_{KL}(B_n^p||B_n^q) \leq \frac{16n \epsilon^2}{3} $$
Substituting this inequality into \eqref{eq:Pin1} proves the result.
\end{proof}

We now utilize the above refined lemma.   
Let $\{\hat{A}_n\}_{n=1}^{\infty}$ be an infinite sequence of estimation functions, as defined in Section \ref{section:estimation-functions}, so that 
\begin{equation} \label{eq:estimation-functions-here}
\hat{A}_n:[0,1) \times \{0,1\}^n \rightarrow [0,1]
\end{equation} 
Fix $p, q \in [0,1]$.  
Let $\{W_n^p\}_{n=1}^{\infty}$ and $\{W_n^q\}_{n=1}^{\infty}$ denote sequences of i.i.d. Bernoulli random variables with $P[W_n^p=1]=p$ and $P[W_n^q=1]=q$, respectively. 
Let $U$ be a random variable that is uniformly distributed over $[0,1)$ and that is independent of both 
$\{W_n^p\}_{n=1}^{\infty}$ and $\{W_n^q\}_{n=1}^{\infty}$. 
For $n \in \{1, 2, 3, \ldots\}$, let $A_n^p$  and $A_n^q$ 
denote the estimates of $p$ and $q$  based on $(U, W_1^p, \ldots, W_n^p)$ and $(U, W_1^q, \ldots, W_n^q)$, respectively: 
\begin{align}
A_n^p &= \hat{A}_n(U, W_1^p, W_2^p, \ldots, W_{n}^p) \quad \forall n \in \{1, 2, 3, \ldots\} \label{eq:Ap-here} \\
A_n^q &= \hat{A}_n(U, W_1^q, W_2^q, \ldots, W_n^q) \quad \forall n \in \{1, 2, 3, \ldots\} \label{eq:Aq-here} 
\end{align}
Fix $\alpha>0$.  We have 
\begin{align*}
\mathbb{E}_p[|A_n^p-p|^{\alpha}] &= \int_0^1 \mathbb{E}_p[|A_n^p-p|^{\alpha} \given U=u] du\\
\mathbb{E}_p[|A_n^q-q|^{\alpha}] &= \int_0^1 \mathbb{E}_q[|A_n^q-q|^{\alpha} \given U=u] du
\end{align*}
where $\mathbb{E}_p[\cdot]$ and $\mathbb{E}_q[\cdot]$ represent expectations with respect
to the probability distributions that form the random vectors $(U, W_1^p, \ldots, W_n^p)$ and $(U, W_1^q, \ldots, W_n^q)$, respectively. Since $U$ is independent of the samples, the conditional 
expectation $\mathbb{E}_p[|A_n^p-p|^{\alpha} \given U=u]$ is 
with respect to the probability measure $B_n^p$ associated only with 
the random vector of i.i.d. Bernoulli-$p$ variables $(W_1^p, \ldots, W_{n}^p)$.  Similarly, 
$\mathbb{E}_q[|A_n^q-q|^{\alpha} \given U=u]$ considers the \emph{same estimation
function} $\hat{A}_n(\cdot)$ but is with respect to the probability measure $B_n^q$ associated only with 
the random vector of i.i.d. Bernoulli-$q$ variables $(W_1^q, \ldots, W_{n}^q)$.   
For $n \in \{1, 2, 3, \ldots\}$, these expectations can be evaluated in terms of the measures $B_n^p$ and $B_n^q$ for which the Pinkser inequality applies. 

The following lemma generalizes Lemma 16 of \cite{hazan-kale-stochastic}, which treats mean square error, to treat general powers of the absolute error. The proof closely follows the structure developed
in \cite{hazan-kale-stochastic} but uses the refined lemma (Lemma \ref{lem:c-lemma} above) in a key place. 

\begin{lem}  \label{lem:Bernoulli-estimation} Fix $\alpha>0$. Fix any sequence of measurable  
estimation functions $\{\hat{A}_n\}_{n=1}^{\infty}$ 
of the form \eqref{eq:estimation-functions-here}.  
Let $p,q$ be probabilities that satisfy $p,q \in [1/4, 3/4]$.  Then for all $n \in \{1, 2, 3, \ldots\}$ we have 
\begin{equation} \label{eq:absolute-regret} 
E_p[|A_n^p-p|^{\alpha}] + E_q[|A_n^q-q|^{\alpha}] \geq \frac{|p-q|^{\alpha}}{2^{1+\alpha}}  \quad \mbox{whenever $|p-q|\leq \frac{1}{2c\sqrt{n}}$}
\end{equation} 
where $A_n^p$ and $A_n^q$ are defined by \eqref{eq:Ap-here} and $c=\sqrt{8/3}$.  
\end{lem}

\begin{proof} Fix $n \in \{1, 2, 3, \ldots\}$. Define $\epsilon = |p-q|$ and 
assume that $\epsilon \leq \frac{1}{2c\sqrt{n}}$.  
It suffices to prove the following claim: 
For all $u \in [0,1)$ we have 
\begin{equation} \label{eq:absolute-regret-claim} 
E_p[|A_n^p-p|^{\alpha} \given U=u] + E_q[|A_n^q-q|^{\alpha} \given U=u] \geq \frac{\epsilon^{\alpha}}{2^{1+\alpha}} 
\end{equation} 
Without loss of generality assume $q\geq p$ so that  $p=q+\epsilon$. 
If $\epsilon=0$ then \eqref{eq:absolute-regret-claim} trivially holds. Assume $\epsilon>0$ and suppose \eqref{eq:absolute-regret-claim} is false (we reach a contradiction). Then 
\begin{equation} \label{eq:sufficiently-small} 
E_p[|A_n^p-p|^{\alpha} \given U=u] + E_q[|A_n^q-q|^{\alpha} \given U=u] < \frac{\epsilon^{\alpha}}{2^{1+\alpha}} 
\end{equation} 
Thus, there is a constant $\theta \in (0,1)$ such that:\footnote{Indeed from \eqref{eq:sufficiently-small}: If $E_p[|A_n^p-p|^{\alpha}]=0$ then any $\theta \in (0,1)$ satisfies \eqref{eq:fails1} and we can choose $\theta \in (0,1)$  sufficiently close to $0$  to ensure \eqref{eq:fails2}.  Else, if $E_q[|A_n^q-q|^{\alpha}]=0$ then any $\theta \in (0,1)$ satisfies
\eqref{eq:fails2} and we can choose $\theta \in (0,1)$ sufficiently close to $1$ to ensure
\eqref{eq:fails1}. Else, define $\theta \in (0,1)$ by $\theta= \frac{E_p[|A_n^p-p|^{\alpha}]}{E_p[|A_n^p-p|^{\alpha}] +E_q[|A_n^q-q|^{\alpha}]}$ and note that both \eqref{eq:fails1}-\eqref{eq:fails2} hold.}
\begin{align}
E_p[|A_n^p-p|^{\alpha} \given U=u]  &<  \frac{\theta \epsilon^{\alpha}}{2^{1+\alpha}} \label{eq:fails1} \\
E_q[|A_n^q-q|^{\alpha} \given U=u] &<\frac{(1-\theta) \epsilon^{\alpha}}{2^{1+\alpha}}\label{eq:fails2} 
\end{align}
Following the technique in \cite{hazan-kale-stochastic},  
applying the Markov inequality to \eqref{eq:fails1} and \eqref{eq:fails2} gives
\begin{align*}
P_p[|A_n^p-p|\geq \epsilon/2 \given U=u] &<  \frac{\theta}{2}   \\
P_q[|A_n^q-q|\geq \epsilon/2 \given U=u] &<  \frac{1-\theta}{2}
\end{align*}
where $P_p[\cdot \given U=u]$ and $P_q[\cdot \given U=u]$ represent probabilities under the probability measures $B_n^p$ and $B_n^q$, respectively.  For simplicity of notation, for the remainder of 
this proof we suppress the explicit ``$U=u$'' conditioning, with the understanding that all probabilities are implicitly conditioned on $U=u$. With this 
simplified notation the above inequalities become
\begin{align}
P_p[|A_n^p-p|\geq \epsilon/2 ] &<  \frac{\theta}{2} \label{eq:fails1-mod} \\
P_q[|A_n^q-q|\geq \epsilon/2  ] &<  \frac{1-\theta}{2} \label{eq:fails2-mod} 
\end{align}
From \eqref{eq:fails1-mod} we obtain
\begin{equation}
P_p[A_n^p\geq p+\epsilon/2] <  \frac{\theta}{2} \label{eq:foo-parallel} 
 \end{equation} 
From \eqref{eq:fails2-mod} we obtain
\begin{align}
&P_q[|A_n^q-q|<\epsilon/2] > \frac{1}{2}+\frac{\theta}{2}   \nonumber \\
&\implies P_q[A_n^q > q-\epsilon/2]> \frac{1}{2}+\frac{\theta}{2}  \nonumber \\
&\overset{(a)}{\implies} P_q[A_n^q > p+\epsilon/2] >\frac{1}{2}+\frac{\theta}{2}   \label{eq:foo} 
\end{align}
where (a) holds because $q= p+\epsilon$.  Now define the set $C\subseteq \{0,1\}^n$ as follows: 
$$ C = \{(x_1, \ldots, x_{n}) \in \{0,1\}^{n} : \hat{A}_n(x_1, \ldots, x_{n}) >p+\epsilon/2\}$$
Then \eqref{eq:foo-parallel} implies $P_p[C] < \theta/2$ and \eqref{eq:foo} implies
$P_q[C] > 1/2 + \theta/2$ and so 
\begin{equation} \label{eq:baz}  
|P_p[C] - P_q[C]|> 1/2
\end{equation} 
Then
\begin{align*}
1/2 &\overset{(a)}{<}  |P_p[C]-P_q[C]| \\
&\overset{(b)}{\leq} v(B_n^p, B_n^q)\\
&\overset{(c)}{\leq} c|p-q|\sqrt{n}\\
&= c\epsilon \sqrt{n}\\
&\overset{(d)}{\leq} 1/2
\end{align*}
where (a) holds by \eqref{eq:baz}; (b) holds by definition of $v(\cdot)$ as the supremum 
absolute error over all possible events (including the event $C$); (c) follows by Lemma \ref{lem:c-lemma}; (d) follows because we have assumed $c\epsilon \sqrt{n}\leq 1/2$. This gives the contradiction.
\end{proof}

\section*{Appendix B --- Proof of Theorem \ref{thm:positive-measure}}

This appendix proves Theorem \ref{thm:positive-measure}. For $\alpha >0$ define 
$$ V_m(\alpha)= \sum_{n=1}^m (1/n)^{\alpha/2} \quad \forall m \in \{1, 2, 3, \ldots\} $$
Notice that
$$ V_m(\alpha) \geq  \int_1^{m+1} (1/t)^{\alpha/2} dt =  \left\{ \begin{array}{ll}
\log(m+1) &\mbox{ if $\alpha=2$} \\
\frac{(m+1)^{1-\alpha/2}-1}{1-\alpha/2} & \mbox{ if $\alpha \neq 2$} 
\end{array}
\right.$$
Fix $\alpha  \in [0,2)$.  Fix any sequence of measurable 
estimation functions $\{\hat{A}_n\}_{n=1}^{\infty}$ of the form \eqref{eq:estimation-functions}. 
Define $\script{Q}$ as the set of all $p \in [1/4, 3/4]$ such that 
\begin{equation}\label{eq:limsup2}
\limsup_{m\rightarrow\infty} \left[\frac{1}{V_m(\alpha)}\sum_{n=1}^m \mathbb{E}_p[|A_n^p-p|^{\alpha}]\right] \geq \frac{1}{c^{\alpha}2^{3+2\alpha}}
\end{equation} 
Let $\mu(\script{Q})$ denote the total Lebesgue measure of the set $\script{Q}$.

We first show that $\script{Q}$ is Lebesgue measurable: Let $\script{X}_n$ be the set of all $2^n$ binary-valued vectors of size $n$. For each random seed $u \in [0,1)$, 
each function $\hat{A}_n(u,x_1, ..., x_n)$ assigns each vector $(x_1, ..., x_n) \in \script{X}_n$ to a 
real number in the interval $[0,1]$. So 
$$\expect{|A_n^p-p|^{\alpha}|U=u} = \sum_{\vec{x} \in \script{X}_n} |\hat{A}_n(u,\vec{x})-p|^{\alpha}\prod_{i=1}^np^{x_i}(1-p)^{1-x_i}$$
and this is a continuous and bounded function of $p \in [0,1]$. Then by the law of iterated expectations
\begin{align*}
 \expect{|A_n^p-p|^{\alpha}} &=  \sum_{\vec{x} \in \script{X}_n} \prod_{i=1}^np^{x_i}(1-p)^{1-x_i}\int_0^1|\hat{A}_n(u,\vec{x})-p|^{\alpha}du
 \end{align*}
where the integral is a well defined real number because $\hat{A}_n(u,\vec{x})$ is a  bounded and measurable function of $u$ and hence $|\hat{A}_n(u,\vec{x})-p|$ can be integrated with respect to the uniform distribution over  $u \in [0,1]$.   It follows that $\expect{|A_n^p-p|^{\alpha}}$ is 
Lebesgue-measurable in $p$.  The weighted sum of Lebesgue-measurable functions is again Lebesgue-measurable, and the set of all $p \in [1/4,3/4]$ under which 
the limsup of these functions exceeds a threshold is again Lebesgue-measurable.  Thus, $\script{Q}$ is indeed Lebesgue-measurable.

We want to show  $\mu(\script{Q})\geq 1/6$.

\begin{proof} (Theorem \ref{thm:positive-measure}) 
Fix $\alpha \in (0, 2]$. 
For each $n \in \{1, 2, 3, \ldots\}$ define 
\begin{equation} \label{eq:epsilon-n} 
\epsilon[n] = \frac{1}{2c\sqrt{n}}
\end{equation}  
with $c=\sqrt{8/3}$.  
It follows by Lemma \ref{lem:Bernoulli-estimation} that if $p,q \in [1/4, 3/4]$ such that $|p-q|= \epsilon[n]$ then
\begin{equation}\label{eq:p-minus-q}
 \mathbb{E}_p[|A_n^p-p|^{\alpha}] + E_q[|A_n^q-q|^{\alpha}] \geq \frac{\epsilon[n]^{\alpha}}{2^{1+\alpha}} 
 \end{equation} 
For each $p \in [1/4, 3/4]$ define: 
\begin{equation}\label{eq:fp}
f_n(p) = \min\left[\mathbb{E}_p[|A_n^p-p|^{\alpha}], \frac{\epsilon[n]^{\alpha}}{2^{1+\alpha}}\right]
\end{equation} 
We first claim that if $p,q \in [1/4, 3/4]$ and $|p-q|=\epsilon[n]$ then:
\begin{equation} \label{eq:f-q-p}
f_n(p) + f_n(q) \geq \frac{\epsilon[n]^{\alpha}}{2^{1+\alpha}} 
\end{equation} 
Indeed, if either $f_n(p)\geq \epsilon[n]^{\alpha}/2^{1+\alpha}$ or $f_n(q)\geq \epsilon[n]^{\alpha}/2^{1+\alpha}$ then the result is trivially true, and otherwise the result holds by \eqref{eq:p-minus-q}.  
Integrating the nonnegative function $f_n(p)$ gives 
\begin{align*}
\frac{1}{1/2}\int_{1/4}^{3/4} f_n(p)dp &\overset{(a)}{=} \int_{1/4}^{3/4} f_n(p)dp + \int_{1/4}^{3/4}f_n(q)dq\\
&\overset{(b)}{\geq} \int_{1/4}^{3/4-\epsilon[n]} f_n(p)dp + \int_{1/4+\epsilon[n]}^{3/4} f_n(q)dq\\
&= \int_{1/4}^{3/4-\epsilon[n]} [f_n(p)+f_n(p+\epsilon[n])]dp \\
&\overset{(c)}{\geq} \int_{1/4}^{3/4-\epsilon[n]} \frac{\epsilon[n]^{\alpha}}{2^{1+\alpha}}dp\\
&= \frac{\epsilon[n]^{\alpha}}{2^{1+\alpha}}(1/2-\epsilon[n])\\
&\overset{(d)}{=} \frac{(1/n)^{\alpha/2}(1- \frac{1}{cn^{1/2}})}{c^{\alpha}2^{2+2\alpha}} 
\end{align*}
where (a) holds by simply doubling the integral; for inequality (b) we note that $\epsilon[n] <1/2$ for all $n \in \{1, 2, 3, \ldots\}$ and 
so $3/4-\epsilon[n]>1/4$ and $1/4+\epsilon[n]<3/4$; 
(c) holds by \eqref{eq:f-q-p}; (d) holds by definition 
$\epsilon[n]$ in \eqref{eq:epsilon-n}. 
Summing over $n \in \{1, \ldots, m\}$ gives 
\begin{equation}\label{eq:above} 
\frac{1}{1/2}\int_{1/4}^{3/4}\left[\sum_{n=1}^m f_n(p)\right]dp  \geq \frac{V_m(\alpha)}{c^{\alpha}2^{2+2\alpha}} - \frac{V_m(\alpha+1)}{c^{\alpha+1}2^{2+2\alpha}} 
\end{equation} 

Now let $Z$ be a random variable that is independent of all else and 
is uniform over $[1/4, 3/4]$. Define $H_m=\sum_{n=1}^m f_n(Z)$. Inequality \eqref{eq:above} 
can be interpreted as
\begin{equation} \label{eq:interpret}
\expect{H_m}  \geq
\frac{V_m(\alpha)}{c^{\alpha}2^{2+2\alpha}} - \frac{V_m(\alpha+1)}{c^{\alpha+1}2^{2+2\alpha}} 
\end{equation}
Inequality  \eqref{eq:fp} implies a deterministic bound on $H_m$: 
\begin{equation*} 
f_n(Z) \leq \frac{\epsilon[n]^{\alpha}}{2^{1+\alpha}}= \frac{(1/n)^{\alpha/2}}{c^{\alpha}2^{1+2\alpha}}
\end{equation*} 
Summing the above over $n \in \{1, \ldots, m\}$ gives
\begin{equation} \label{eq:H-bound}
H_m \leq \frac{V_m(\alpha)}{c^{\alpha}2^{1+2\alpha}}
\end{equation} 
Define $A_m(\alpha)$ as the following event: 
$$ A_m(\alpha) = \left\{H_m \geq \frac{V_m(\alpha)}{c^{\alpha}2^{3+2\alpha}}\right\}$$
Define $A_m(\alpha)^c$ as the complement of this event. Then
\begin{equation}
\expect{H_m} =\expect{H_m \given A_m(\alpha)^c}(1-P[A_m(\alpha)])+ \expect{H_m \given A_m(\alpha)}P[A_m(\alpha)] \label{eq:pre-part} 
\end{equation} 
However
\begin{align}
\expect{H_m|A_m(\alpha)^c} &\leq \frac{V_m(\alpha)}{c^{\alpha}2^{3+2\alpha}} \label{eq:part1}\\
\expect{H_m|A_m(\alpha)} &\leq  \frac{V_m(\alpha)}{c^{\alpha}2^{1+2\alpha}}\label{eq:part2}
\end{align}
where \eqref{eq:part1} holds by definition of the event $A_m(\alpha)$ and \eqref{eq:part2} uses the deterministic upper
bound on $H_m$ in \eqref{eq:H-bound}.   Substituting \eqref{eq:part1} and \eqref{eq:part2} into \eqref{eq:pre-part} gives  
\begin{align*}
\expect{H_m}  \leq \frac{V_m(\alpha)}{c^{\alpha}2^{3+2\alpha}} + \frac{(3/4)V_m(\alpha)}{c^{\alpha}2^{1+2\alpha}}P[A_m(\alpha)]
\end{align*}
Substituting this inequality into \eqref{eq:interpret} gives
$$ \frac{V_m(\alpha)}{c^{\alpha}2^{2+2\alpha}} - \frac{V_m(\alpha+1)}{c^{\alpha+1}2^{2+2\alpha}}  \leq  \frac{V_m(\alpha)}{c^{\alpha}2^{3+2\alpha}} + \frac{(3/4)V_m(\alpha)}{c^{\alpha}2^{1+2\alpha}}P[A_m(\alpha)]$$
Rearranging terms of the above inequality gives 
$$P[A_m(\alpha)]\geq \frac{1}{3}  - \frac{(2/3)V_m(\alpha+1)}{cV_m(\alpha)}\quad \forall m \in \{1, 2, 3, \ldots\} \nonumber $$
Taking limits as $m\rightarrow\infty$ and using the fact that $\alpha \in (0,2]$ implies  $\lim_{m\rightarrow\infty} \frac{V_m(\alpha+1)}{V_m(\alpha)}=0$ yields: 
$$\limsup_{m\rightarrow\infty} P[A_m(\alpha)] \geq 1/3$$
 Thus
\begin{align*}
1/3&\leq \limsup_{m\rightarrow\infty} P[A_m(\alpha)] \\
&\overset{(a)}{\leq} \limsup_{m\rightarrow\infty} P[\cup_{i=m}^{\infty} A_i(\alpha)] \\
&\overset{(b)}{=} \lim_{m\rightarrow\infty} P[\cup_{i=m}^{\infty} A_i(\alpha)] \\
&\overset{(c)}{=} P[A_m(\alpha) \quad i.o]
\end{align*} 
where (a) holds because $A_m(\alpha) \subseteq \cup_{i=m}^{\infty} A_i(\alpha)$; (b) and 
(c) hold by monotonicity of probability: 
\begin{align*}
&\left(\cup_{i=m}^{\infty} A_i(\alpha)\right)\searrow \{A_m(\alpha) \quad i.o\}\\
\implies & \lim_{m\rightarrow\infty}P[\cup_{i=m}^{\infty} A_i(\alpha)] = P[A_m(\alpha) \quad i.o.]
\end{align*}
where the notation ``i.o.''  
represents ``infinitely often,'' that is, $P[A_m(\alpha) \quad i.o.]$ 
is the probability  that the event $A_m(\alpha)$ occurs for an infinite number of 
 indices $m$. Thus 
 $$1/3 \leq P[A_m(\alpha) \quad i.o.]  $$
 However, by definition of the events $A_m(\alpha)$ we have 
 $$ \{A_m(\alpha) \quad i.o.\} \subseteq \left\{\limsup_{m\rightarrow\infty} \frac{H_m}{V_m(\alpha)} \geq \frac{1}{c^{\alpha}2^{3+2\alpha}}\right\}$$ 
 Thus
\begin{equation} 
1/3 \leq P\left[\limsup_{m\rightarrow\infty} \frac{H_m}{V_m(\alpha)}\geq \frac{1}{c^{\alpha}2^{3+2\alpha}}\right] \label{eq:thus} 
 \end{equation} 
Finally we note by definition of $f_n(p)$ in \eqref{eq:fp} that 
 $$ H_m = \sum_{n=1}^m f_n(Z)\leq \sum_{n=1}^m E_Z[|A_n^Z-Z|^{\alpha}] $$
Substituting this into \eqref{eq:thus} gives
$$ P\left[\limsup_{m\rightarrow\infty} \frac{\sum_{n=1}^m E_Z[|A_n^Z-Z|^{\alpha}]}{V_m(\alpha)}\geq \frac{1}{c^{\alpha}2^{3+2\alpha}} \right] \geq 1/3$$
Since $Z$ is chosen uniformly over the size-$(1/2)$ interval $[1/4, 3/4]$ it follows
that the measure of all values $p \in [1/4, 3/4]$ for which the above $\limsup$ inequality 
holds is at least $1/6$, that is, $\mu(\script{Q})\geq 1/6$. 
\end{proof} 

\section*{Appendix C --- Tightness of Bounds for Bernoulli Estimation} 

Let $\{W_n\}_{n=1}^{\infty}$ be i.i.d. Bernoulli random variables with $P[W_n=1]=p$, where $p \in [0,1]$ is an unknown parameter. Theorem \ref{thm:positive-measure} establishes a lower bound on regret 
for arbitrary (biased or unbiased) estimators of $p$. This appendix shows that the simple (and unbiased) estimators $\{A_n\}_{n=1}^{\infty}$ defined by 
\begin{equation} \label{eq:simple} 
A_n = \frac{1}{n}\sum_{i=1}^n W_i \quad \forall n \in \{1, 2, 3, \ldots\}
\end{equation} 
achieve the regret bounds of Theorem \ref{thm:positive-measure} to within a constant factor. 
Fix $\alpha \in (0,2]$. 
Define 
$$ regret_n = \sum_{m=1}^n\expect{|A_m-p|^{\alpha}}$$

\begin{itemize} 

\item (Case $\alpha =2$)
For each $n \in \{1, 2, 3, ...\}$ we have from \eqref{eq:simple}: 
\begin{equation} \label{eq:MSE-average}
 \expect{(A_n-p)^2} = \expect{\left(\frac{1}{n}\sum_{m=1}^n(W_n-p)\right)^2} = \frac{p(1-p)}{n}
\end{equation} 
and so 
\begin{align*}
regret_n &= p(1-p)\sum_{m=1}^n\frac{1}{m} \\
&\overset{(a)}{\leq} \frac{1}{4}\sum_{m=1}^n \frac{1}{m} \\
&\leq \frac{1}{4}\left[ 1+ \int_{1}^{n} \frac{1}{t}dt\right]\\
&=\frac{1}{4} + \frac{1}{4}\log(n) \quad \forall n \in \{1, 2, 3, \ldots\} 
\end{align*}
where (a) holds because $p(1-p)\leq 1/4$ whenever $p \in [0,1]$.  Thus 
$$ \limsup_{n\rightarrow\infty} \left[\frac{regret_n}{\log(n+1)}\right] \leq \frac{1}{4}$$
This value $1/4$ is within a factor of $85.4$ from the  lower bound $3/2^{10}$ of Theorem \ref{thm:positive-measure}.

\item (Case $0< \alpha < 2$) 
Fix $m \in \{1, 2, 3, \ldots\}$.  We have
\begin{align}
\expect{|A_m-p|^\alpha} &= \expect{(|A_m-p|^2)^{\alpha/2}}\nonumber \\
&\overset{(a)}{\leq} \expect{|A_m-p|^2}^{\alpha/2}\nonumber \\
&\overset{(b)}{=} (\frac{p(1-p)}{m})^{\alpha/2} \label{eq:case2-reuse} 
\end{align}
where (a) holds by Jensen's inequality on the concave function $x^{\alpha/2}$; (b) holds by \eqref{eq:MSE-average}.  Thus
\begin{align*}
regret_n &= \sum_{m=1}^n \expect{|A_m-p|^{\alpha}}\\
&= (p(1-p))^{\alpha/2}\sum_{m=1}^n \frac{1}{m^{\alpha/2}} \\
&\leq (1/4)^{\alpha/2}\left[1 + \int_1^n 1/t^{\alpha/2}dt\right]\\
&= \left[\frac{(1/2)^{\alpha}}{1-\alpha/2}\right]\left[n^{1-\alpha/2}- \frac{\alpha}{2}\right] \quad \forall n \in \{1, 2, 3, \ldots\} 
\end{align*}
Thus
$$ \limsup_{n\rightarrow\infty}\left[ \frac{regret_n}{(n+1)^{1-\alpha/2}-1}\right] \leq \frac{(1/2)^{\alpha}}{1-\alpha/2}$$
This value is within a factor $\frac{2^{3 + \frac{5\alpha}{2}}}{3^{\alpha/2}}$ of the lower bound $\frac{1}{c^{\alpha}2^{3+2\alpha}(1-\alpha/2)}$ of Theorem \ref{thm:positive-measure} (where $c = \sqrt{8/3}$).  For $\alpha=1$ the 
constant factor gap is $32\sqrt{2/3} \approx 26.12789$.   For $\alpha\rightarrow 0^+$ the 
constant factor gap converges to $8$.
\end{itemize} 

\bibliographystyle{unsrt}
\bibliography{../../../../../../../latex-mit/bibliography/refs}
\end{document}